\documentclass[10pt]{amsart}
\usepackage{amsmath,amssymb,paralist, comment}
\usepackage{latexsym,graphicx,hyperref}
\usepackage[dvipsnames]{xcolor}
\usepackage{tikz}
\usetikzlibrary{arrows,shapes,trees,automata}
\usetikzlibrary{braids}
\usepackage[normalem]{ulem}
\theoremstyle{plain}
   \newtheorem{theorem}{Theorem}[section]
   \newtheorem{proposition}[theorem]{Proposition}
   \newtheorem{prop}[theorem]{Proposition}
   \newtheorem{lemma}[theorem]{Lemma}
   \newtheorem{corollary}[theorem]{Corollary}
    \newtheorem{cor}[theorem]{Corollary}

   \newtheorem*{theorem*}{Theorem}
\theoremstyle{definition}
   \newtheorem{definition}[theorem]{Definition}
   \newtheorem{example}[theorem]{Example}

   \newtheorem{remark}[theorem]{Remark}

\numberwithin{equation}{section}

\newcommand\Symm{\mathfrak{S}}

\newcommand\rank{\operatorname{rank}}
\newcommand\codim{\operatorname{codim}}
\newcommand\Par{\operatorname{Par}}
\newcommand\wt{\operatorname{wt}}
\newcommand\fix{\operatorname{fix}}
\newcommand{\defeq}{\overset{\text{def}}{=\hspace{-4pt}=}}
\newcommand{\type}[1]{\mathrm{#1}}
\newcommand{\BBB}{\mathcal{B}}

\newcommand\CC{{\mathbb{C}}}
\newcommand\ZZ{{\mathbb{Z}}}

\newcommand{\lR}{\ell_{R}}
\newcommand{\aaa}{\mathbf{a}}

\newcommand{\cyc}{\operatorname{cyc}}
\newcommand{\Parmax}{\operatorname{Par}_{\mathrm{max}}}
\newcommand{\id}{\mathrm{id}}
\newcommand{\pt}[1]{[\![ #1 ]\!]}

 \newcommand{\yg}[1]{{\color{Green} #1}} 
 \newcommand{\bl}[1]{{\color{Plum} #1}}
 \newcommand{\br}[1]{{\color{blue} #1}}

 \newcommand{\yo}[1]{{\color{RedOrange} #1}}

\newtheoremstyle{TheoremNum}
        {}{}                            
        {\itshape}                      
        {}                              
        {\bfseries}                     
        {.}                             
        { }                             
        {\thmname{#1}\thmnote{ \bfseries #3}}
    \theoremstyle{TheoremNum}
    \newtheorem{corr}[theorem]{Corollary}
    \newtheorem{thmm}[theorem]{Theorem}

\begin{document}

\title{The Hurwitz action in complex reflection groups}
\author{Joel Brewster Lewis and Jiayuan Wang}
\address{Department of Mathematics \\ George Washington University \\ Washington, DC, USA 20052}
\email{\{jblewis, j453w588\}@gwu.edu}

\maketitle

\begin{abstract}
We enumerate Hurwitz orbits of shortest reflection factorizations of an arbitrary element in the infinite family $G(m, p, n)$ of complex reflection groups.  As a consequence, we characterize the elements for which the action is transitive and give a simple criterion to tell when two shortest reflection factorizations belong to the same Hurwitz orbit.  We also characterize the quasi-Coxeter elements (those with a shortest reflection factorization that generates the whole group) in $G(m, p, n)$.
\end{abstract}


\section{Introduction}

Given an arbitrary group $G$, there is an action of the braid group 
\[
\BBB_{n} = \left\langle \sigma_1, \ldots, \sigma_{n - 1} \; \middle| \begin{array}{ll}
\sigma_i \sigma_{i + 1} \sigma_i = \sigma_{i + 1} \sigma_i \sigma_{i + 1} & \text{for } i = 1, \ldots, n - 2 \\
\sigma_i \sigma_j = \sigma_j \sigma_i & \text{if } |i - j| > 1
\end{array}
\right\rangle
\] on the set $G^n$ of $n$-tuples of elements of $G$: the generator $\sigma_i$ acts via a \emph{Hurwitz move} 
\[
\begin{array}{ccccl}
 \big(t_{1}, \; \ldots , \; t_{i-1}, & t_{i}, & t_{i+1}, & t_{i + 2}, \; \ldots , \; t_{n}\big) & \overset{\sigma_i}{\longmapsto}\\
 \big(t_{1}, \; \ldots , \; t_{i-1}, & t_{i+1}, & t_{i+1}^{-1}t_{i}t_{i+1}, & t_{i+2}, \; \ldots , \; t_{n}\big) & ,
\end{array}
\]
preserving the product of the tuple.  It is easy to check that this extends to an action of $\BBB_n$, which is called the \emph{Hurwitz action}.
The case that $G$ is a reflection group and the factors $t_j$ are reflections is of particular interest.  Notably, one ingredient in Bessis's proof \cite{Bessis} of the $K(\pi, 1)$ property for complements of finite complex reflection arrangements is the following transitivity property.
\begin{theorem}[{Bessis \cite[Prop.~7.6]{Bessis}}]
\label{thm:bessis}
If $G$ is a well-generated complex reflection group and $c$ is a Coxeter element in $G$, the Hurwitz action is transitive on minimum-length reflection factorizations of $c$.
\end{theorem}
The broader context for this result is the so-called \emph{dual Coxeter theory}, developed by Bessis \cite{Bes03} and Brady and Watt \cite{Br01, BW02} for real reflection groups and then extended to the complex case.  Working from this perspective, Baumeister--Gobet--Roberts--Wegener \cite{BGRW} gave a complete characterization of those elements in a finite real reflection group with the property that the Hurwitz action is transitive on their minimum-length reflection factorizations.  (The precise statement of their result may be found below as Theorem~\ref{thm:BGRW qc}.)

Our goal in the present paper is to extend the work of Baumeister et al.\ to the complex case.  We have three main results.  First, in Theorem~\ref{main theorem}, we give an exact enumeration of Hurwitz orbits of minimum-length reflection factorizations for an arbitrary element in the infinite family $G(m, p, n)$ of ``combinatorial'' complex reflection groups.  As a consequence, we obtain a characterization (Corollary~\ref{cor transitive}) of the elements in $G(m, p, n)$ such that the Hurwitz action is transitive on their minimum-length reflection factorizations.  Second, in Theorem~\ref{thm:invariants}, we solve the ``inverse problem'' for Hurwitz equivalence on minimum-length reflection factorizations in $G(m, p, n)$, showing that two such factorizations of the same element belong to the same Hurwitz orbit if and only if they generate the same subgroup of $G(m, p, n)$.  Third, in Theorem~\ref{cor:qc length n}, extend a result Baumeister et al.\ by showing that for an arbitrary finite complex reflection group $G$ of rank $n$, if $g$ in $G$ has reflection length $n$ and has a minimum-length factorization that generates $G$, then in fact all minimum-length factorizations of $g$ generate $G$.

The plan of the paper is as follows: in Section~\ref{sec:background}, we give background on complex reflection groups, reflection factorizations, and the Hurwitz action, and establish notation and conventions that will be used throughout the paper.  Section~\ref{sec:main} is devoted to the statement and proof of our main enumerative theorem (Theorem~\ref{main theorem}).  In Section~\ref{sec:invariants}, we build on the tools developed in Section~\ref{sec:main} in order to solve the inverse problem for Hurwitz equivalence in $G(m, p, n)$ (Theorem~\ref{thm:invariants}).  In Section~\ref{sec:qc}, we introduce the \emph{quasi-Coxeter property}, which plays a key role in the work of Baumeister et al.\ in the real case.  We characterize the elements in $G(m, p, n)$ with this property, discuss its connection with the transitivity of the Hurwitz action, and prove that a weak form of the property implies a stronger form for any complex reflection group (Theorem~\ref{cor:qc length n}).  Finally, in Section~\ref{sec:open}, we discuss several open problems.

\subsection*{Acknowledgements}
We thank Alejandro Morales and Christian Stump for helpful discussions during the preparation of this paper.  We are deeply indebted to Theo Douvropoulos for many comments and suggestions, and especially for his advice and assistance with the computations in Sections~\ref{sec:invariants} and~\ref{sec:qc}.

Work of the first author was supported in part by an ORAU Powe award, a Simons collaboration grant (634530), and the GW University Facilitating Fund.

\section{Background}
\label{sec:background}

\subsection{Conventions}
\label{sec:conventions}

Throughout this paper, $m$, $p$, and $n$ will represent positive integers with $p \mid m$. Since $p$ divides $m$, the cyclic group $\ZZ/m\ZZ$, whose elements are equivalence classes of integers modulo $m$, has a unique subgroup $p\ZZ / m\ZZ \cong \ZZ / (m/p)\ZZ$ of order $m/p$: it consists of those equivalence classes whose elements are divisible by $p$.  We may write $k \equiv 0 \pmod{p}$ to indicate that an element $k$ in $\ZZ/m\ZZ$ belongs to this subgroup, and either $k = 0$ or $k \equiv 0 \pmod{m}$ to indicate that $k$ is the identity in $\ZZ/m\ZZ$.  As in the previous sentence, we do not distinguish notationally between the integer $k$ and its equivalence class modulo $m$.

Given a collection of $k_1, k_2, \ldots, k_n$ of elements of $\ZZ/m\ZZ$, they generate a subgroup $k\ZZ/m\ZZ$ for some $k$ (the minimal subgroup that contains all of them).  If we take $k_1, \ldots, k_n$ to be any representatives of their equivalence classes, then the smallest positive representative $k$ of $k\ZZ$ is $k = \gcd(m, k_1, \ldots, k_n)$.  All greatest common divisors that appear in this paper (particularly, as in Definition~\ref{def r star}) will be meant in this sense.

Again throughout the paper, $\omega = \exp(2\pi i / m)$ will denote a fixed primitive $m$th root of unity.  The collection $\{ \omega^i \colon i \in \ZZ\}$ of all $m$th roots of unity is in natural bijection with $\ZZ/m\ZZ$, and the $(m/p)$th roots of unity correspond exactly to the elements $k \equiv 0 \pmod{p}$.

\subsection{Complex reflection groups}
For a general reference on the material in this section, see \cite{LehrerTaylor}.
Given a finite-dimensional complex vector space $V$, a \emph{reflection} is a linear transformation $t: V \rightarrow V$ whose fixed space $\operatorname{ker}(t-1)$ is a hyperplane (i.e., has codimension $1$), and a finite subgroup $G$ of $GL(V)$ is called a \emph{complex reflection group} if $G$ is generated by its subset $R$ of reflections.\footnote{In the literature, it is sometimes required \emph{a priori} that reflections have finite order or that complex reflection groups live in the unitary group on $V$.  The former condition is implied by the fact that $G$ is finite, and the latter may be recovered by a standard averaging trick \cite[Lem.~1.3]{LehrerTaylor}.} Complex reflection groups were classified by Shephard and Todd~\cite{ST54}: every complex reflection group is a direct product of irreducibles, and every irreducible is isomorphic to a group of the form
\[
G(m, p, n) \defeq \left\{\begin{array}{l} n \times n \text{ monomial matrices whose nonzero entries are}\\m\text{th} \text{ roots of unity with product a } \frac{m}{p}\text{th} \text{ root of unity}\end{array} \right\}
\]
for positive integers $m$, $p$, $n$ with $p \mid m$ or to one of $34$ exceptional examples.

When $p < m$, the group $G(m, p, n)$ contains two types of reflections: for $a = 0, 1, \ldots, m - 1$, the \emph{transposition-like reflections}
\[
\left[
\begin{smallmatrix}
1&&&&&&&&&&\\
&\ddots &&&&&&&&&\\
&&1&&&&&&&& \\
&&&&&&&\overline{\omega}^{a} &&&\\
&&&&1&&&&&&\\
&&&&&\ddots&&&&&\\
&&&&&&1&&&&\\
&&&\omega^{a}&&&&&&&\\
&&&&&&&&1&&\\
&&&&&&&&&\ddots&\\
&&&&&&&&&&1
\end{smallmatrix}
\right]
\]
of order $2$, and for $b = p, 2p, \ldots, \left(\frac{m}{p} - 1\right)p$ the \emph{diagonal reflections}
\begin{equation}
\label{eq:diagonal reflection}
\begin{bmatrix}
1&&&&&&\\
&\ddots&&&&&\\
&&1&&&&\\
&&&\omega^{b}&&&\\
&&&&1&&\\
&&&&&\ddots&\\
&&&&&&1
\end{bmatrix}
\end{equation}
of various orders.  The group $G(m, m, n)$ contains only the transposition-like reflections.

One recovers the infinite families of real reflection groups (finite Coxeter groups) as the following special cases: the group $G(1, 1, n)$ is the symmetric group $\Symm_n$ (of Coxeter type $\type{A}_{n - 1}$); $G(2, 1, n)$ is the hyperoctahedral group of signed permutations (type $\type{B}_n$/$\type{C}_n$); $G(2, 2, n)$ is the group of even-signed permutations (type $\type{D}_n$); and $G(m, m, 2)$ is the dihedral group of order $2 \times m$ (type $\type{I}_2(m)$).

For every $m$, $p$, $n$, there is a natural projection map 
\[
\pi: G(m, p, n) \twoheadrightarrow G(1, 1, n) = \Symm_n
\]
defined as follows: for $g \in G(m, p, n)$, the matrix $\pi(g)$ is the result of replacing every root of unity in the matrix of $g$ with $1$.  The resulting permutation is called the \emph{underlying permutation} of $g$.  It will be convenient to use the following, more compact, notation for elements of $G(m, p, n)$: one writes $g = [u; (a_1, \ldots, a_n)]$ where $u = \pi(g)$ and $a_j \in \ZZ/m\ZZ$ is the exponent of $\omega = \exp(2\pi i/m)$ in the nonzero entry of the $j$th column of $g$.  For example, in this notation we have in $G(30, 5, 6)$ that
\begin{equation}
\label{eq:example element}
\left[\begin{smallmatrix}
&\omega^{21} &&&&\\
\omega&&&&&\\
&&&&\omega^{2}&\\
&&\omega^{2}&&&\\
&&&\omega^{3}&&\\
&&&&&\omega^{6}
\end{smallmatrix}\right]
= [214536 ; (1, 21, 2, 3, 2, 6)] = 
[(12)(345)(6) ; (1, 21, 2, 3, 2, 6)].
\end{equation}
This notation reveals that $G(m, 1, n)$ has the structure of a \emph{wreath product}, namely, $G(m, 1, n) \cong (\ZZ/m\ZZ) \wr \Symm_n $, with multiplication given by
\[
[u; (a_1, \ldots, a_n)] \cdot [v; (b_1, \ldots, b_n)] = \left[uv; (a_{v(1)} + b_1, \ldots, a_{v(n)} + b_n)\right].
\]
In the particular case of transposition-like reflections, we abbreviate the notation even further, setting
\[
[(i\, j); a] \defeq
[(i\, j); (0, \ldots, 0, a, 0, \ldots, 0, -a, 0, \ldots, 0)].
\]

Given an element $g = [u; (a_1, \ldots, a_n)]$ of $G(m, p, n)$, the value $a_i$ is called the \emph{weight} of $i$.  Further, for any subset $S \subseteq \{1, \ldots, n\}$, we define $\sum_{i \in S} a_i$ to be the \emph{weight of $S$ in $g$}.  This notion will be particularly relevant when $S$ is the set of entries of a cycle of $g$, or when $S = \{1, \ldots, n\}$ and $a_1 + \ldots + a_n$ is the \emph{weight of $g$}.  For instance, every transposition-like reflection has weight $0$, while the diagonal-like reflection in \eqref{eq:diagonal reflection} has weight $b$.  In this language, an element of $G(m,1,n)$ belongs to $G(m,p,n)$ if and only if its weight is a multiple of $p$.

\subsection{Shi's formula for reflection length}
\label{sec:reflection length}

Fix a complex reflection group $G$ with reflections $R$.  Since $G$ is a reflection group, every element $g$ of $G$ can be written as a product of reflections.  If $f = (t_1, \ldots, t_\ell)$ is a tuple of reflections such that $g = t_1 \cdots t_\ell$, we say that $f$ is a \emph{(reflection) factorization} of $g$.  We say that a reflection factorization $F$ of $g$ is \emph{shortest}, \emph{minimum}, or \emph{of minimum length} if there is no reflection factorization of $g$ using fewer reflections, and we define the \emph{reflection length} $\lR(g)$ of $g$ to be the length
\[
\lR(g) \defeq \min \{\ell: g=t_{1}t_{1}\cdots t_{\ell} \text{ for some } t_{i}\in R\}
\]
of the shortest factorizations. 

When $G$ is a finite \emph{real} reflection group, the reflection length of an element can be interpreted geometrically: one has $\lR(g) = \codim_V(\fix(g))$, where $\fix(g) = \{v \in V \colon g(v) = v\}$ is the fixed space of $g$ \cite[Lem.~2]{Carter}. The same is true in $G(m, 1, n)$, but \emph{not} in any of the other irreducible complex reflection groups \cite{FG, Shi}.  In \cite{Shi}, Shi gave a combinatorial formula for reflection length in the group $G(m, p, n)$ that we describe next.

For an element $g\in G(m, p, n)$, we say that a \emph{cycle} of $g$ is a cycle of the underlying permutation $\pi(g)$, and we denote by $\cyc(g)$ the number of cycles of $g$.  A \emph{cycle partition} $\Pi$ of $g$ is a set partition of the set $\{ C_{1}, \ldots, C_{\cyc(g)} \}$ of cycles of $g$ such that for every part in $\Pi$, the corresponding cycle weights sum to $0\pmod{p}$.  (Such partitions always exist because the weight of $g$ is $0 \pmod{p}$.) For example, the element
\[
g= \begin{bmatrix} \omega^{2}&&&\\ &\omega^{2}&&\\ &&\omega^{2}&\\ &&&\omega^{2}\end{bmatrix} = [\id; (2, 2, 2, 2)] \in G(4,4,4)
\]
has four cycle partitions:
\[
\Par(g) = \Big\{\pt{(1)(2) \mid (3)(4)}, \pt{(1)(3) \mid (2)(4)}, \pt{(1)(4) \mid (2)(3)}, \pt{(1)(2)(3)(4)} \Big\}.
\]
Observe that the set of cycle partitions depends on the choice of the group containing $g$: if we were to view the same matrix $g$ as an element of $G(4, 2, 4)$ then $\Par(g)$ would consist of all fifteen set partitions of the four cycles.

Given a partition $\Pi$ of an element $g \in G(m, p, n)$, let $|\Pi|$ denote the number of parts of $\Pi$ and let $v_m(\Pi)$ denote the number of parts of $\Pi$ of weight $0$ (not just $0 \pmod{p}$).  The \emph{value} $v(\Pi)$ of a cycle partition $\Pi$ is
\[
v(\Pi) \defeq |\Pi| + v_m(\Pi).
\]
A partition is \emph{maximum} if its value is the largest among the values of all possible cycle partitions of $g$ (relative to the given $m, p$), and we denote by $\Parmax(g)$ the set of maximum cycle partitions of $g$.  For example, with $g = [\id; (2, 2, 2, 2)] \in G(4, 4, 4)$ as above, the three partitions into two parts have value $4$, while the partition in one part has value $2$, and so \[
\Parmax(g)=\Big\{\pt{(1)(2) \mid (3)(4)}, \pt{(1)(3) \mid (2)(4)}, \pt{(1)(4) \mid (2)(3)}\Big\}.
\]

\begin{theorem}[{\cite[Thm.~4.4]{Shi}}]\label{shi}
Given an element $g \in G(m, p, n)$, its reflection length is 
\[
\lR(g)= n + \cyc(g)-v(\Pi),
\] 
where $\cyc(g)$ is the number of cycles of $g$ and $\Pi\in \Parmax(g)$ is any maximum cycle partition of $g$.
\end{theorem}

\begin{example}
\label{eg:shi's formula}
The element $g = [(12)(345)(6) ; (1, 21, 2, 3, 2, 6)]$ of $G(30, 5, 6)$ shown in \eqref{eq:example element} has three cycles, of weights $22$, $7$, and $6$.  The unique partition of the cycles of $g$ in which all parts have weight $0 \pmod{5}$ is the one-part partition $\Pi = \pt{(12)(345)(6)}$, of value $v(\Pi) = 1$.  Therefore $\lR(g) = 6 + 3 - 1 = 8$.

If instead $g = [\id; (2, 2, 2, 2)]$ is the element of $G(4, 4, 4)$ mentioned just before the statement of Theorem~\ref{shi}, then $\Parmax(g)$ contains three partitions, each of value $4$.  Thus for this element $\lR(g) = 4 + 4 - 4 = 4$.
\end{example}

\begin{remark}[Theorems~2.1 and~3.1 in~\cite{Shi}]
\label{rmk:shi special cases}
In the particular cases $p = 1$ and $p = m$, the formula in Theorem~\ref{shi} can be simplified. 

When $p = m$, every part in every cycle partition $\Pi$ must have weight $0\pmod{m}$, so $v_{m}(\Pi) = |\Pi|$ and $v(\Pi)= |\Pi| + v_{m}(\Pi) = 2|\Pi|$. Thus, in $G(m, m, n)$ we have that $\lR(g) = n + \cyc(g)- 2\max_{\Pi \in \Par(g)} |\Pi|$. 

When $p = 1$, every partition of the cycles is a cycle partition.  It is shown below (as part of the proof of Corollary~\ref{cor special}, in Section~\ref{sec:proofs}) that a partition $\Pi$ is maximum if and only if each part contains either a single cycle or two cycles of nonzero weights that sum to $0$.  It follows that in $G(m, 1, n)$ we have $\lR(g)= n - \#\{\text{cycles in } g\text{ of weight } 0\}$.

Upon taking $m = 1$, both formulas correctly collapse to the symmetric group formula $\lR(g) = n - \cyc(g)$.
\end{remark}

\begin{remark}
From a computational perspective, determining the reflection length of an element in $G(m, 1, n)$ is easy -- naively computing cycle weights and counting how many are $0$ (as in Remark~\ref{rmk:shi special cases}) is polynomial time.  Unfortunately, the same is not true in general, or even in $G(m, m, n)$.  In particular, testing whether an element $g \in G(m, m, n)$ has $\lR(g) < n + \cyc(g) - 2$ amounts to asking whether there is any nontrivial subset of the cycle weights that sums to $0$ modulo $m$.  This is known in the literature as the \emph{modular subset sum problem} (see, e.g., \cite{subset sum}).  The non-modular version of this question (in which the cycle weights are integers, rather than integers modulo $m$) was one of Karp's original NP-complete problems (he called in the \emph{knapsack problem}, although it is now more commonly known as the \emph{subset sum problem}) \cite{NP complete}.  There is a standard reduction from the non-modular to modular versions of the question, namely, by setting $m = 2S + 1$ where $S$ is the sum of absolute values of the given collection of numbers; thus modular subset sum is also NP-complete.

(The non-modular version of subset sum appears in computing reflection length in the affine symmetric group \cite[App.~A]{LMPS}, which is a generic cover in a certain sense of the groups $G(m, m, n)$ \cite{ShiGeneric, JBL}.)
\end{remark}

\subsection{The Hurwitz action}

As mentioned in the introduction, it is easy to check that Hurwitz moves satisfy the braid relations
\begin{align*}
\sigma_i \sigma_{i + 1} \sigma_i(t_1, \ldots, t_n) & = \sigma_{i + 1} \sigma_i \sigma_{i + 1}(t_1, \ldots, t_n) && \text{for } i = 1, \ldots, n - 2 \text{ and}\\
\sigma_i \sigma_j(t_1, \ldots, t_n) &= \sigma_j \sigma_i(t_1, \ldots, t_n) && \text{if } {|i - j|} > 1,
\end{align*}
and consequently induce an action of the full braid group.  In particular, the action of the inverse of a generator is
\[
\begin{array}{ccccl}
 \big(t_{1}, \; \ldots, \; t_{i-1}, & t_{i}, & t_{i+1}, & t_{i + 2}, \; \ldots, \; t_{n}\big) & \overset{\sigma_i^{-1}}{\longmapsto}\\
\big(t_{1}, \; \ldots, \; t_{i-1}, & t_{i}t_{i+1}t_{i}^{-1}, & t_{i}, & t_{i+2}, \; \ldots, \; t_{n}\big) &.
\end{array}
\]
We will refer variously to \emph{Hurwitz orbits} (the orbits under the Hurwitz action), \emph{Hurwitz equivalence} (the induced equivalence relation, of belonging to the same orbit), and \emph{Hurwitz paths} (sequences of Hurwitz moves connecting factorizations in the same orbit) when speaking of the orbit structure of this action.

The Hurwitz action was introduced by Hurwitz~\cite{Hur91}, in the context of his study of the coverings of the Riemann sphere with given monodromy.  In that setting, the group $G$ is the symmetric group $\Symm_n$, and the allowed factors are transpositions.  The structure of the Hurwitz orbits on transposition factorizations in $\Symm_n$ was essentially completely analyzed by Kluitmann.
\begin{theorem}[{Kluitmann~\cite[Thm.\ 1]{Kluitmann}}]\label{Kluitmann}
The set of all transposition factorizations $(t_{1}, t_{2}, \ldots, t_{k})$ of an element $w\in \Symm_n$ such that $\langle t_{1}, t_{2},\ldots, t_{k}\rangle = \Symm_{n}$ forms a single orbit under the Hurwitz action.
\end{theorem}

Kluitmann's approach employs a graph-theoretic model for transposition factorizations that goes back to D\'enes \cite{Denes}. This model may be easily extended to the classical families of Coxeter groups, and more generally to $G(m, p, n)$ (as in, for example, \cite{LR, WangShi}); we will make use of it below in Section~\ref{sec:main}.

We end this section by mentioning a few other results on the Hurwitz action related to Bessis's theorem (Theorem~\ref{thm:bessis}) on minimum-length reflection factorizations of Coxeter elements.  (For a detailed discussion of the definition and properties of Coxeter elements in the complex setting, see \cite{RRS}.)  The analogous result was proved for (not necessarily finite) Coxeter groups by Igusa and Schiffler~\cite{IS10}, with a short and self-contained proof in \cite{BDSW}.  This has been extended to arbitrary-length reflection factorizations of Coxeter elements in Coxeter groups (first finite \cite{LR}, then in general \cite{WY19}) and to the infinite family of complex reflection groups \cite{JBL}.  Finally, in \cite{Wegener}, Wegener extended one direction of the main result of \cite{BGRW} to the case of affine Coxeter groups.

\section{The main theorem}
\label{sec:main}

This section is devoted to the proof of our first main theorem.  In order to state it, we need one further piece of terminology.  (The reader may wish to recall our convention for greatest common divisors from Section~\ref{sec:conventions}.)

\begin{definition}\label{def r star}
Let $g\in G(m, p, n)$, let $\Pi$ be a cycle partition of $g$, and let $B$ be a part in $\Pi$. Suppose that the weights of the cycles in $B$ are $(k_{1}, k_{2}, \ldots, k_{|B|})$. Define $r(B)=\gcd(m, k_{1},k_{2}, \ldots, k_{|B|})$.
\end{definition}

\begin{theorem}\label{main theorem}
Given an element $g \in G(m, p, n)$, the number of Hurwitz orbits of shortest reflection factorizations of $g$ is
\[
\sum_{\Pi\in\Parmax(g)} \prod_{B\in \Pi} (r(B))^{|B|-1}.
\]
\end{theorem}

\begin{example}
For example, the element $g = [\id; (2, 2, 2, 2)] \in G(4, 4, 4)$ discussed above has twelve Hurwitz orbits of minimum-length reflection factorizations, including $4 = 2^{2 - 1} \times 2^{2 - 1}$ from each of its three maximum cycle partitions.
\end{example}

As an immediate consequence of Theorem~\ref{main theorem}, we will derive the following characterization of elements for which the Hurwitz acts transitively on the shortest reflection factorizations.

\begin{cor}\label{cor transitive}
Let $g\in G(m,p,n)$. The shortest reflection factorizations of $g$ form a single orbit under the Hurwitz action if and only if $\Parmax(g) = \{\Pi\}$ is a singleton set and either ${|B|} = 1$ or $r(B)=1$ for every part $B\in\Pi$.
\end{cor}

We mention two particularly nice special cases.

\begin{cor}\label{cor single cycle}
If $g\in G(m, p, n)$ has a single cycle, then the shortest reflection factorizations of $g$ form a single orbit under the Hurwitz action.
\end{cor}

\begin{cor}\label{cor special}
Let $g\in G(m,1,n)$.  Then the shortest factorizations of $g$ form a single orbit under the Hurwitz action if and only if $g$ does not have two cycles of nonzero weight whose weights sum to $0\pmod{m}$.
\end{cor}

We divide the proof into four steps: first, we define a standard form of factorizations and show that every factorization is Hurwitz-equivalent to a factorization in standard form; next, we present a sufficient condition for two factorizations in standard form to be Hurwitz equivalent, by explicitly constructing a Hurwitz path between them; then we show that this condition is also necessary, by constructing an invariant that distinguishes Hurwitz orbits; and lastly, we conclude with the proof of Theorem~\ref{main theorem} and its corollaries.

\subsection{Standard forms}

In this section, we define a standard form for factorizations in $G(m, p, n)$ and show that every shortest factorization is Hurwitz-equivalent to a factorization in standard form.  It is easiest to describe the standard form in terms of a graph object associated to factorizations.

\begin{definition}\label{def graph}
Given a reflection factorization $f = (t_{1}, \ldots,  t_{\ell})$ of an element $g = t_1 \cdots t_\ell \in G(m,p,n)$, the \emph{factorization graph} of $f$ is the graph $\Gamma_f = (V,E)$ on (labeled) vertex set $V=\{1, \dots, n\}$ with (labeled) edges $E = \{e_1, \ldots, e_\ell\}$ defined as follows: if $t_k$ has underlying permutation $(i\, j)$ then $e_k$ joins vertices $i$ and $j$, while if $t_k$ is diagonal with nonzero weight in position $i$ then $e_k$ is a loop at vertex $i$.
\end{definition}

\begin{example}\label{example std}
Consider the element $g = 
[(12)(345)(6); (1, 21, 2, 3, 2, 6)]$ in $G(30, 5, 6)$ from Example~\ref{eg:shi's formula}, with factorization
\begin{multline}
\label{eq:example factorization}
f = \Big([(13); 1], \; [(13); 23], \; [(36); 0], \; [(36); 29], \\
 [\id; (0,0,0,0,0,5)], \; [(12); 1], \; [(34); 2], \; [(45); 3]\Big).
\end{multline}

The factorization graph $\Gamma_f$ is 
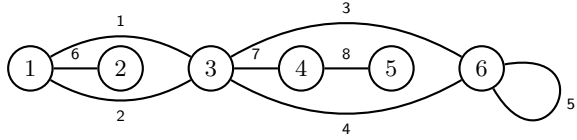
\begin{center}
 \begin{tikzpicture}[-,>=stealth',auto,node distance=1.2cm,
  thick,main node/.style={circle,draw, font=\sffamily\small\bfseries}]

  \node[main node] (1) {$1$};
  \node[main node] (2) [right of=1] {$2$};
  \node[main node] (3) [right of=2] {$3$};
  \node[main node] (4) [right of=3] {$4$};
  \node[main node] (5) [right of=4] {$5$};
  \node[main node] (6) [right of=5] {$6$};
  
  \path[every node/.style={font=\sffamily\small}]
    (1) edge node [above] {\tiny 6} (2)
    (1) edge[bend left] node [above] {\tiny 1} (3)
    (1) edge[bend right] node [below] {\tiny 2} (3)
    (3) edge[bend left] node [above] {\tiny 3} (6)
    (3) edge[bend right] node [below] {\tiny 4} (6)
    (3) edge node [above] {\tiny 7} (4)
    (4) edge node [above] {\tiny 8} (5)
    (6) edge [out=370, in=300, looseness=10] node [right] {\tiny 5} (6);
 \end{tikzpicture} \raisebox{.25in}{.}
\end{center}
\end{example}

Fix an element $g$ in $G(m, p, n)$ and a reflection factorization $f = (t_1, \ldots, t_\ell)$ of $g$.  The connected components of the graph $\Gamma_f$ form a set partition of $\{1, \ldots, n\}$.  In fact, this set partition corresponds to a cycle partition of $g$: the action of each transposition $\pi(t_i)$ respects the set partition, so $\pi(g)$ must as well, and consequently all elements of a single cycle belong to the same part.  We denote by $\Pi_f$ the cycle partition induced in this way by the factorization $f$.  If $B$ is a part in $\Pi_f$, we denote by $\Gamma_f|_B$ the connected component of $\Gamma_f$ that corresponds to $B$ and by $f|_B$ the subsequence of $f$ whose factors correspond to the edges in $\Gamma_f|_B$.

The next result is a trivial observation concerning the relationship between the factors in different connected components of $\Gamma_f$; its proof is left to the reader as an exercise.

\begin{prop}
\label{prop:components commute}
Suppose that $f$ is a reflection factorization of an element $g \in G(m, p, n)$, and that $B \neq B'$ are parts in the cycle partition $\Pi_f$.  Then every reflection in $f|_B$ commutes with every reflection in $f|_{B'}$.  
Moreover, the product $g|_B$ of the factors in $f|_B$ has the following form: its nontrivial cycles are the cycles in $B$, and the weight of $j$ in $g|_B$ agrees with its weight in $g$ if $j$ is in the support of $B$, and is $0$ otherwise.
\end{prop}

Next, we give a more detailed description of the structure of the graphs that come from minimum-length factorizations.

\begin{proposition}\label{prop correct number}
Suppose that $f$ is a shortest factorization of an element $g \in G(m, p, n)$, and let $\Pi_f$ be the cycle partition induced by the factorization graph $\Gamma_f$.  Then $\Pi_f$ is a maximum cycle partition of $g$.

Moreover, if $B$ is a part of $\Pi_f$ that contains $c = |B|$ cycles of $g$, supported on a total of $k$ elements of $\{1, \ldots, n\}$, then $\Gamma_f|_B$ contains $c+k-2$ edges that join distinct vertices and either one loop (if the weight of $B$ in $g$ is not $0$) or no loops (if the weight of $B$ is $0$).
\end{proposition}
\begin{proof}
The result is implicit in the proofs of Lemmas~4.2 and~4.3 and Theorem~4.4 in~\cite{Shi}.
\end{proof}

\begin{example}
Consider the factorization graph $\Gamma_f$ in Example~\ref{example std}. Since $\Gamma_f$ has a single connected component, the cycle partition $\Pi_f$ has a single part that contains all three cycles in $g$, with $c = 3$ and $k = 6$.  In $\Gamma_f$, there are $3+6-2 = 7$ edges joining distinct vertices.  Since the weight of the element being factored is not $0$, there is exactly one loop, for a total of eight edges.  This is the same as the reflection length of $g$ (as computed in Example~\ref{eg:shi's formula}).
\end{example}

\begin{remark}
\label{rmk:construct a std factorization}
Given any element $g \in G(m, p, n)$ and any cycle partition $\Pi$ of $g$, it is easy to construct a factorization $f$ such that $\Pi_f = \Pi$, as follows:
choose a part $B$ in $\Pi$, and suppose that $B = \{C_1 = (v_{1, 1} \cdots), \ldots, C_c = (v_{c, 1} \cdots)\}$ contains $c$ cycles, whose weights in $g$ are $k_1, \ldots, k_c$.  Begin with a product of $2(c - 1)$ transposition-like factors 
\begin{multline*}
[(v_{1, 1} \, v_{2, 1}); a_1] \cdot [(v_{1, 1} \, v_{2, 1}); a_1 + k_1] \cdot [(v_{2, 1} \, v_{3, 1}); a_2] \cdot [(v_{2, 1} \, v_{3, 1}); a_2 + k_1 + k_2] \cdots
\\
\cdots [(v_{c - 1, 1} \, v_{c, 1}); a_{c - 1}] \cdot [(v_{c - 1, 1} \, v_{c, 1}); a_{c - 1} + k_1 + \ldots + k_{c - 1}]
\end{multline*}
for arbitrary $a_i \in \ZZ/m\ZZ$.
If the weight of $B$ is nonzero, multiply on the right by a diagonal reflection of weight $\wt(B)$ in position $v_{c, 1}$.  The resulting product is a diagonal element in which the weight of $v_{i, 1}$ is $k_i$ for $i = 1, \ldots, c$ and the other weights are $0$.  Next, for each cycle $C_i$, use a tree of $|C_i| - 1$ transposition-like factors to form the cycle; the weights are uniquely determined by the sequence of underlying transpositions and the requirement that the product be the (correctly weighted) cycle in $g$.  Then do the same for the other parts of $\Pi$.
\end{remark}

Inspired by the construction of Remark~\ref{rmk:construct a std factorization}, we are now prepared for the key definition of this section.

\begin{definition}\label{def std form}
Given a shortest reflection factorization $f$ of an element $g\in G(m,p,n)$, with factorization graph $\Gamma_f$ and cycle partition $\Pi_f$.  Let $B_1$, $B_2$, \ldots, $B_{|\Pi_f|}$ be the parts of $\Pi_f$.  For each part $B_i \in \Pi_f$, let $C_{i, 1}$, \ldots, $C_{i, |B_i|}$ be the cycles of $g$ contained in $B_i$ and for each $j$, let $v_{i,j}$ be the smallest element of $\{1, \ldots, n \}$ permuted by $C_{i,j}$.  Without loss of generality, assume that the indices of the $B_i$ and $C_{i, j}$ have been chosen so that $v_{i,1} < v_{i,2} < \ldots < v_{i, |B_i|}$ and $v_{1, 1} < v_{2, 1} < \ldots < v_{|\Pi_f|, 1}$.  We say that $f$ is in \emph{standard form} if the following conditions are met:
\begin{enumerate}[(a)]
\item if $i' < i$ then every factor in $f|_{B_{i'}}$ precedes all factors in $f|_{B_i}$;
\item for each $i$, the first $2|B_i|-2$ factors $t_1, \ldots, t_{2|B_i| - 2}$ in $f|_{B_i}$ have underlying transpositions
\begin{align*}
\pi(t_{1}) & = \pi(t_{2}) = (v_{i,1} \, v_{i,2}),\\
\pi(t_{3}) & = \pi(t_{4}) = (v_{i,2} \, v_{i,3}),\\
& \vdots\\
\pi(t_{2|B_i|-3}) & = \pi(t_{2|B_i| - 2}) = (v_{i, |B_i| - 1} \, v_{i,|B_i|});
\end{align*}
and
\item for each $i$, if there is a loop in $\Gamma_f|_{B_i}$ then it is $e_{2|B_i| - 1}$ and is incident to vertex $v_{i, |B_i|}$.
\end{enumerate}
\end{definition}

\begin{example}\label{example non std}
Once again consider the factorization $f$ from Example~\ref{example std}. It is a shortest factorization since $f$ has $\lR(g)=8$ factors. The factorization graph $\Gamma_f$ has only one component, corresponding to a cycle partition in which  with three cycles $(12)$, $(345)$, and $(6)$ are all placed in the same part.  Thus condition (a) of the definition is satisfied vacuously.  By definition, $v_{1, 1} = 1$, $v_{1, 2} = 3$ and $v_{1, 3} = 6$ are the smallest elements in their cycles.  The first two factors in $f$ have underlying permutation $(v_{1, 1}\, v_{1, 2}) = (13)$ and the next two factors have underlying permutation $(v_{1, 2}\, v_{1, 3}) = (36)$, so condition (b) is satisfied.  There is a loop, and it is the fifth factor and is attached to vertex $v_{1, 3} = 6$, so condition (c) is satisfied. Therefore, $f$ is in standard form.

\end{example}


By Proposition~\ref{prop:components commute}, condition (a) in Definition~\ref{def std form} is ``for free'': we can commute factors from different components past each other without affecting conditions (b) and (c).
The next proposition spells out in great detail the structure of factorizations in standard form.

\begin{prop}\label{prop correct form}
Let $g\in G(m, p,n)$ with a standard form factorization $f$ and associated graph $\Gamma_f$ and partition $\Pi_f$. For each part $B$ of $\Pi_f$, the connected component $\Gamma_f|_{B}$ has the following structure: (1) it contains a sequence of vertices $v_{1}$, \ldots, $v_{|B|}$ such that $e_{2i - 1}$ and $e_{2i}$ join $v_{i}$ to $v_{i + 1}$; (2) if it contains a loop, the loop is $e_{2|B| - 1}$ and is attached to vertex $v_{|B|}$; and (3) removing the edges in (1) and (2) from $\Gamma_f|_B$ leaves a forest with $|B|$ components, each of which contains exactly one of the vertices $v_1, \ldots, v_{|B|}$.
\end{prop}

\begin{proof}
Parts (1) and (2) follow immediately from parts (b) and (c) of Definition~\ref{def std form}, respectively.  We now consider part (3).

Let $f|_B = (t_1, t_2, \ldots)$ be the subfactorization of $f$ whose reflections correspond to edges in $\Gamma_f|_B$.  By part (a) of Definition~\ref{def std form}, this is a consecutive subsequence of $f$.  By Proposition~\ref{prop:components commute}, its product $g|_B$ has underlying permutation whose nontrivial cycles are the cycles in $B$, and whose weights agree with $g$ on the support of $B$ and are $0$ otherwise.  

By part (1), for $i = 1, \ldots, |B| - 1$ we have that the underlying permutations $\pi(t_{2i - 1})$ and $\pi(t_{2i})$ are both equal to the transposition $(v_i \, v_{i + 1})$.  Consequently, the underlying permutation $\pi(t_{2i - 1}t_{2i})$ of their product is the identity.  Likewise, if there is a loop $e_{2|B| - 1}$, then the reflection $t_{2|B| - 1}$ is diagonal and so $\pi(t_{2|B| - 1}) = \id$.  It follows that removing these factors from the factorization yields a shorter reflection factorization of an element with underlying permutation $\pi(g|_B)$, with $c = |B|$ cycles.

Let $k$ be the size of the support of $B$.  By Proposition~\ref{prop correct number}, $f|_B$ contains either $c + k - 1$ or $c + k - 2$ edges, of which we remove either $2c - 1$ or $2c - 2$, respectively.  Thus exactly $k - c$ edges remain after deletion.  By elementary graph theory, a graph with $k - c$ edges and $k$ vertices has at least $c$ connected components.  On the other hand, the edges correspond to a transposition factorization of the permutation $\pi(g|_B)$ with $c$ cycles, and the entries of any given cycle must be in the same connected component, so there are at most $c$ connected components.  Combining these two statements, we have that the deleted graph has exactly $c$ connected components, and in fact the vertex set of each component is the support of a cycle in $B$.  Thus, each component contains exactly one of the vertices $v_1, \ldots, v_c$.  Finally, by elementary graph theory, every graph with $k$ vertices, $k - c$ edges, and $c$ connected components is a forest. 
\end{proof}

\begin{lemma}\label{lemma any in p}
For any $g \in G(m, p, n)$ and any minimum-length factorization $f$ of $g$, there is a standard form factorization of $g$ that is Hurwitz-equivalent to $f$.
\end{lemma}
\begin{proof}
It is easy to see that the Hurwitz action preserves the connected components of $\Gamma_f$.  Since factors from different components of $\Gamma_f$ commute (Proposition~\ref{prop:components commute}), it suffices to consider the case that $\Gamma_f$ is connected (equivalently, that $\Pi_f$ is a one-part partition).

Let $f'$ be a factorization of $g$ in standard form.  (This exists by, for example, the construction in Remark~\ref{rmk:construct a std factorization}.)  By Propositions~\ref{prop correct number} and~\ref{prop correct form}, both $f$ and $f'$ have $n + c - 2$ transposition-like factors and either both have no diagonal factors or both have one diagonal factor.  In the case that both have no diagonal factors (i.e., when $\wt(g) = 0$), it follows immediately from Kluitmann's theorem (Theorem~\ref{Kluitmann}) that the projected factorizations $\pi(f)$ and $\pi(f')$ are Hurwitz-equivalent in $\Symm_n$.  In the case that both factorizations have a diagonal factor, may reduce to Kluitmann's theorem by the easy observation that if two tuples of elements in an arbitrary group are Hurwitz-equivalent, then they remain Hurwitz-equivalent if we insert a copy of the identity in an arbitrary position in each tuple.  Thus, in either case we have that there is some braid (i.e., product of Hurwitz moves) $\beta$ such that $\beta(\pi(f)) = \pi(f')$.  Hurwitz moves commute with projection, so $\pi(\beta(f)) = \pi(f')$.  The definition of standard form depends only on the projection; since $f'$ is in standard form, it follows that $\beta(f)$ is, as well.   Then $\beta(f)$ is the desired 
factorization of $g$.
\end{proof}

\subsection{Hurwitz paths between standard form factorizations}

In this section, we construct explicit sequences of Hurwitz moves connecting certain standard form factorizations to each other.  The main result of the section is Lemma~\ref{lem h path for all}, which gives a criterion for two factorizations in standard form to be Hurwitz-equivalent.

In the first part of the section, we concentrate on connected factorizations that consist only of the edges that appear in parts (b) and (c) of Definition~\ref{def std form}.

\begin{definition}\label{def doubled path}
For $n \geq 1$, we say that a factorization in $G(m, p, n)$ is a \emph{doubled path} if it is of the form
\[
\Big([(12);a_{1}], [(12);b_{1}], [(23); a_{2}], [(23); b_{2}], \ldots, [(n - 1 \; n); a_{n - 1}], [(n - 1 \; n); b_{n - 1}] \Big)
\]
or
\[
\Big([(12);a_{1}], \; [(12);b_{1}], \ldots, [(n - 1 \; n); a_{n - 1}], \; [(n - 1 \; n); b_{n - 1}], \; [\id; (0,\ldots,0,d)] \Big).
\]
\end{definition}

We construct a collection of explicit sequences of Hurwitz moves that are well-behaved when restricted to doubled paths, and use them as building blocks to prove the necessary Hurwitz equivalences.  As a first step, we establish some additional terminology.

By multiplying out, it's easy to see that every doubled path in $G(m, p, n)$ is a factorization of a diagonal element $g = [\id; (k_1, \ldots, k_n)] \in G(m,p, n)$ of weight~$0$ (if no diagonal element is present) or~$d$, and that for each $i$ one has $b_i = a_i + k_1 + k_2 + \ldots + k_i$.  Consequently, given the product $g$, the $a_i$ determine the entire factorization.  This suggests the following definition.

\begin{definition}\label{def pair weight and constant}
Let 
$
f= \Big([(12);a_{1}], [(12); b_{1}], [(23); a_{2}], [(23); b_{2}], \ldots \Big)
$
be a doubled path, factoring an element $g = [\id; (k_1, \ldots, k_n)] \in G(m,p, n)$. Define the \emph{pair weight} of the $i$-th pair of factors (i.e., with underlying transposition $(i\; i+1)$) to be $a_{i}$ and the corresponding \emph{pair difference} to be $d_i \defeq b_i - a_i = k_{1}+\cdots+k_{i}$.
\end{definition}

We consider three families of operations on doubled paths.  The first family is extremely simple.

\begin{prop}\label{prop operation one}
Let $f$ be the doubled path in $G(m, p, n)$ that has pair weights $(a_1, \ldots, a_{n - 1})$ and pair differences $(d_1, \ldots, d_{n - 1})$.  Then $\sigma_{2i-1}(f)$ is also a doubled path.  Moreover, $\sigma_{2i-1}(f)$ has the same pair weights and pair differences as $f$,  except that the $i$-th pair weight of $\sigma_{2i - 1}(f)$ is $a_i + d_i$.
\end{prop}

\begin{proof}
The full generality may be captured by considering the case $n = 2$, $i = 1$, when $f$ has only one pair of factors, 
\[
f= \Big([(1 2); a_1], [(1 2); b_1]\Big),
\]
where $b_1 = a_1 + d_1$.  Applying $\sigma_1$ gives 
\begin{align*}
\Big(\br{[(1 2); a_1]}, \yg{[(1 2); b_1]}\Big) & \overset{\sigma_{1}}{\longrightarrow} \Big({\yg{[(1 2); b_1]}}, {\br{[(1 2);2b_1-a_1]}} \Big) \\
& = \Big({\yg{[(1 2); a_1 + d_1]}}, {\br{[(1 2); b_1 + d_1]}} \Big)= \sigma_1(f).
\end{align*}
The result has new pair weight $a_1 + d_1$ and the same pair difference $b_1-a_1 = d_1$, as claimed.
\end{proof}

The second family of operations is more complex, and is indexed by a pair $i < j$ of numbers in $\{1, \ldots, n - 1\}$.

\begin{definition}\label{def operation three}
For any pair of indices $i, j$ with $1 \leq i < j \leq n-1$, 
define $\tau_{i,j}$ to be the following sequence of Hurwitz moves: 
\begin{multline*}
\tau_{i,j} \defeq \sigma_{2j-2} \circ \sigma_{2j-1} \circ \sigma_{2j-3}^{-1} \circ \sigma_{2j-2}^{-1}\circ\cdots \circ 
 \sigma_{2i+2}\circ\sigma_{2i+3}\circ\sigma_{2i+1}^{-1}\circ\sigma_{2i+2}^{-1}\circ {} \\ 
 \sigma_{2i}^{-1}\circ \sigma_{2i+1}^{-1}\circ \sigma_{2i-1}^{-1}\circ \sigma_{2i}^{-1}\circ \sigma_{2i}^{-1}\circ \sigma_{2i+1}^{-1}\circ \sigma_{2i-1}^{-1}\circ \sigma_{2i}^{-1} \circ {} 
 \\ \sigma_{2i+2}\circ\sigma_{2i+1}\circ \sigma_{2i+3}^{-1}\circ \sigma_{2i+2}^{-1}\circ\cdots \circ \sigma_{2j-2}\circ\sigma_{2j-3}\circ\sigma_{2j-1}^{-1}\circ \sigma_{2j-2}^{-1}.
\end{multline*}
\end{definition}

The associated braids are illustrated in Figure~\ref{fig:braid 1}.

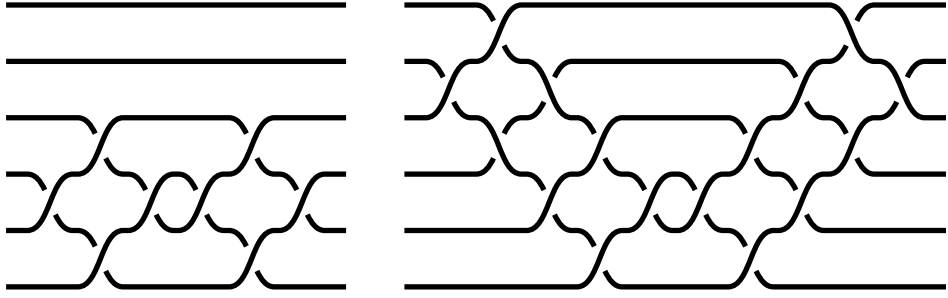
\begin{figure}
\begin{center}
\rotatebox{90}{\begin{tikzpicture}
\pic[braid/.cd, number of strands=6, 
line width=2pt 
,gap =0.15, width = .75cm, height=.67cm,
, name prefix=braid,]{braid={a_2^{-1} a_1^{-1}-a_3^{-1} a_2^{-1} a_2^{-1} a_1^{-1}-a_3^{-1} a_2^{-1}}};
\end{tikzpicture}}%
\qquad
\rotatebox{90}{\begin{tikzpicture}
\pic[braid/.cd, number of strands=6, 
line width=2pt 
,gap =0.15, width = .75cm, height=.67cm,
, name prefix=braid,]{braid={a_4 a_5-a_3^{-1} a_4^{-1}-a_2^{-1} a_1^{-1}-a_3^{-1} a_2^{-1} a_2^{-1} a_1^{-1}-a_3^{-1} a_2^{-1}-a_4 a_3-a_5^{-1} a_4^{-1}}};
\end{tikzpicture}}
\end{center}
\caption{The braids corresponding to $\tau_{1, 2}$ and $\tau_{1, 3}$ when $n = 4$.}
\label{fig:braid 1}
\end{figure}

\begin{prop}\label{prop operation three}
Let $f$ be the doubled path in $G(m, p, n)$ that has pair weights $(a_1, \ldots, a_{n - 1})$ and pair differences $(d_1, \ldots, d_{n - 1})$.  Then $\tau_{i, j}(f)$ is also a doubled path.
Moreover, $\tau_{i, j}(f)$ has the same pair weights and pair differences as $f$, except that the $i$th pair weight of $\tau_{i, j}(f)$ is $a_i + d_j$ and the $j$th pair weight of $\tau_{i, j}(f)$ is $a_j + d_i$.
\end{prop}

\begin{proof}
First consider the case $i = 1$ and $j = 2$, with
\[
\tau_{1, 2} = \sigma_{2}^{-1}\circ\sigma_{3}^{-1} \circ \sigma_{1}^{-1}\circ \sigma_{2}^{-1}\circ \sigma_{2}^{-1} \circ \sigma_{3}^{-1} \circ \sigma_{1}^{-1} \circ\sigma_{2}^{-1}.
\]
In this case, the pair weights for $f= \left([(1 2); a_1], [(1 2); b_1], [(2 3);a_2], [(2 3); b_2], \ldots\right)$ are $(a_1, a_2, \ldots)$ and the pair differences are $(d_1, d_2, \ldots) = (b_1 -a_1, b_2-a_2, \ldots)$. We compute 
\begin{align*}
f & =\Big(\yg{[(1 2); a_1]} , \yo{[(1 2); b_1]} ,\bl{[(2 3);a_2]} , \br{[(2 3); b_2]}, \ldots\Big) \\
& \overset{\sigma_2^{-1}}{\longrightarrow} \Big(\yg{[(1 2); a_1]} , \bl{[(1 3); a_2 + b_1]} , \yo{[(1 2);b_1]} ,\br{[(2 3); b_2]}, \ldots \Big) \\
& \overset{\sigma_{1}^{-1}}{\longrightarrow} \Big(\bl{[(2 3); a_2+d_1]} , \yg{[(1 2); a_1]} , \yo{[(1 2);b_1]} , \br{[(2 3); b_2]}, \ldots\Big) \\
& \overset{\sigma_{3}^{-1}}{\longrightarrow} \Big(\bl{[(2 3); a_2+d_1]} , \yg{[(1\, 2); a_1]} , \br{[(13);b_2 + b_1]} , \yo{[(1 2);b_1]}, \ldots\Big) \\
& \overset{\sigma_{2}^{-1}}{\longrightarrow} \Big(\bl{[(2 3); a_2+d_1]} ,\br{[(23);b_2+d_1]} , \yg{[(1 2); a_1]} , \yo{[(1 2);b_1]}, \ldots\Big) \\
& \overset{\sigma_{2}^{-1}}{\longrightarrow} \Big(\bl{[(2 3); a_2+d_1]} , \yg{[(13);a_1+b_2 + d_1]} , \br{[(23); b_2+d_1]} , \yo{[(1 2);b_1]}, \ldots \Big)\\
& \overset{\sigma_{1}^{-1}}{\longrightarrow} \Big(\yg{[(12);a_1+d_2]} , \bl{[(2 3); a_2+d_1]} , \br{[(2 3); b_2+d_1]} , \yo{[(1 2);b_1]}, \ldots\Big)\\
& \overset{\sigma_{3}^{-1}}{\longrightarrow} \Big(\yg{[(12);a_1+d_2]} , \bl{[(2 3); a_2+d_1]} , \yo{[(13);b_1+b_2+ d_1]} , \br{[(2 3); b_2+d_1]}, \ldots\Big)\\
& \overset{\sigma_{2}^{-1}}{\longrightarrow} \Big(\yg{[(12);a_1+d_2]} , \yo{[(1 2);b_1 +d_2]} , \bl{[(2 3); a_2+d_1]} , \br{[(2 3); b_2+d_1]}, \ldots \Big)
=  \tau_{1,2}(f).
\end{align*}
The pair differences of $\tau_{1,2}(f)$ are $(b_1-a_1, b_2-a_2, \ldots) = (d_1, d_2, \ldots)$ and the pair weights are $(a_1 + d_2, a_2 + d_1, \ldots)$, as desired.

The case that $j = i + 1$ for $i > 1$ is identical to the previous case except for the indices.

For general $j > i$, one may show by induction that, after applying the first $4(j - i - 1)$ Hurwitz moves in $\tau_{i, j}$ to $f$, the resulting factorization has the following four factors as entries in position $2i - 1$, $2i$, $2i + 1$, and $2i + 2$:
\[
\Big( [(i \; i + 1); a_i], \;
 [(i \; i + 1); b_i], \;
 [(i + 1 \; j+1); \square], \;
 [(i + 1 \; j+1); \triangle] \Big),
\]
where $\square = a_{j}+ b_{j-1}+\ldots+ b_{i+2}+b_{i+1}$ and $\triangle = \square + d_j$.
Applying the same calculation as above replaces these four factors with
\[
\Big( [(i\; i+1);a_i + d_j], [(i\; i+1); b_i + d_j], [(i+1\; j+1); \square+ d_i], [(i+1\; j+1); \triangle +d_i] \Big),
\]
and the final $4(j - i - 1)$ Hurwitz moves restore the intermediate pairs and place the factorization back in standard form.

For example, in the case $i = 1$, $j = 3$, starting from 
\[
f= \Big(\br{[(1 2); a_1]}, \br{[(1 2); b_1]}, \bl{[(2 3);a_2]}, \bl{[(2 3); b_2]}, \yg{[(3 4); a_3]}, \yg{[(3 4);b_3]}\Big),
\]
the first four Hurwitz moves produce
\begin{align*}
f & \overset{\sigma_{4}^{-1}}{\longrightarrow} \Big(\br{[(1 2); a_1]}, \br{[(1 2); b_1]}, \bl{[(2 3);a_2]},\yg{[(24);a_3+b_2]},\bl{[(2 3); b_2]},\yg{[(3 4);b_3]}\Big)\\
 & \overset{\sigma_{5}^{-1}}{\longrightarrow} \Big(\br{[(1 2); a_1]}, \br{[(1 2); b_1]},\bl{[(2 3);a_2]},\yg{[(2 4);a_3+b_2]},\yg{[(24); b_3+b_2]},\bl{[(2 3); b_2]}\Big)\\
 &\overset{\sigma_{3}}{\longrightarrow}\Big(\br{[(1 2); a_1]}, \br{[(1 2); b_1]}, \yg{[(2 4); a_3+b_2]}, \bl{[(3 4);a_3+d_2]},\yg{[(2 4); b_3+b_2]},\bl{[(2 3);b_2]}\Big)\\
 & \overset{\sigma_{4}}{\longrightarrow}\Big(\br{[(1 2); a_1]}, \br{[(1 2); b_1]}, \yg{[(2 4); a_3+b_2]},\yg{[(2 4); b_3+b_2]}, \bl{[(2 3); a_2+d_3]},\bl{[(2 3);b_2]}\Big).
\end{align*}
Then applying $\tau_{1, 2}$ affects only the first four factors:
\begin{multline*}
\Big(\br{[(1 2); a_1]}, \br{[(1 2); b_1]}, \yg{[(2 4); a_3+b_2]},\yg{[(2 4); b_3+b_2]}, \ldots \Big) \overset{\tau_{1,2}}{\longrightarrow} {} \\
\Big(\br{[(1 2); a_1+d_3]}, \br{[(1 2); b_1+d_3]}, \yg{[(2 4); a_3+b_2+d_1]}, \yg{[(2 4);b_3+b_2+d_1]}, \ldots\Big).
\end{multline*}
And finally the last four Hurwitz moves leave the first two factors untouched while restoring the original middle factors:
\begin{align*}
\Big(\ldots, \yg{[(2 4); a_3+b_2+d_1]}, \yg{[(2 4);b_3+b_2+d_1]}, \bl{[(2 3); a_2+d_3]}, \bl{[(2 3);b_2]}\Big) &\overset{\sigma_{4}^{-1}}{\longrightarrow} \\
\Big(\ldots, \yg{[(2 4); a_3+b_2+d_1]}, \bl{[(3 4);a_3+d_2+d_1]}, \yg{[(2 4);b_3+b_2+d_1]}, \bl{[(23);b_2]}\Big) & \overset{\sigma_{3}^{-1}}{\longrightarrow} \\
\Big(\ldots,\bl{[(23);a_2]}, \yg{[(2 4); a_3+b_2+d_1]}, \yg{[(2 4);b_3+b_2+d_1]}, \bl{[(2 3);b_2]} \Big) & \overset{\sigma_{5}}{\longrightarrow} \\
\Big(\ldots,\bl{[(23);a_2]},\yg{[(24); a_3+b_2+d_1]},\bl{[(2 3);b_2]},\yg{[(3 4);b_3+d_1]} \Big) & \overset{\sigma_{4}}{\longrightarrow} \\ \Big(\ldots, \bl{[(2 3);a_2]}, \bl{[(2 3); b_2]}, \yg{[(3 4); a_3+d_1]}, \yg{[(3 4);b_3+d_1]} \Big) & = \tau_{1,3}(f).
 \end{align*} 
In $\tau_{1,3}(f)$, the pair weights are $(a_1+d_3, a_2, a_3+d_1)$ and the pair differences are $(b_1-a_1, b_2-a_2, b_3-a_3)= (d_1, d_2, d_3)$, as claimed.
\end{proof}

Finally, we introduce a third family of operations that will be of use in the case that the factorization includes a diagonal reflection.

\begin{definition}\label{def operation four}
Given $i$ where $1\leq i \leq n-1$, define $\gamma_{i}$ to be the following sequence of Hurwitz moves:
\begin{multline*}
\gamma_{i} \defeq \sigma_{2n-2}\circ\sigma_{2n-3}^{-1}\circ\cdots\circ\sigma_{2i+2}\circ\sigma_{2i+1}^{-1}\circ {}\\
\sigma_{2i}\circ\sigma_{2i-1}\circ\sigma_{2i-1}\circ\sigma_{2i}\circ {}\\
\sigma_{2i+1}\circ\sigma_{2i+2}^{-1}\circ \cdots \circ \sigma_{2n-3}\circ\sigma_{2n-2}^{-1}.
\end{multline*}
\end{definition}

The associated braid is illustrated in Figure~\ref{fig:braid 3}.

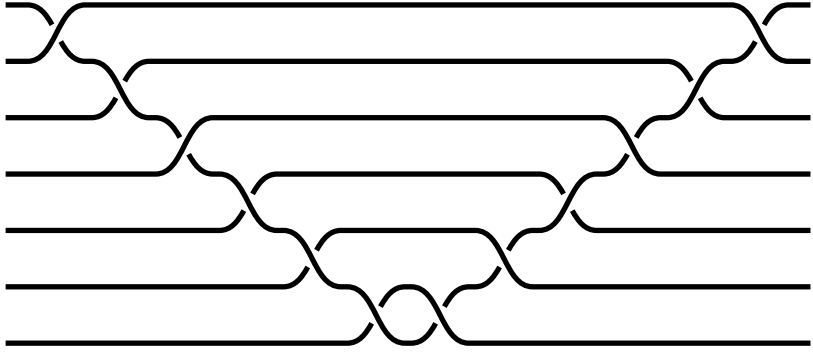
\begin{figure}
\begin{center}
\rotatebox{90}{\begin{tikzpicture}
\pic[braid/.cd, number of strands=7, 
line width=2pt 
,gap =0.1, width = .75cm, height=.85cm,
, name prefix=braid,]{braid={a_6 a_5^{-1} a_4 a_3^{-1} a_2 a_1 a_1 a_2 a_3 a_4^{-1} a_5 a_6^{-1}}};
\end{tikzpicture}}
\end{center}
\caption{The braid corresponding to $\gamma_{1}$ when $n = 4$.} 
\label{fig:braid 3}
\end{figure}

\begin{prop}\label{prop operation four}
Let $f$ be the doubled path in $G(m, p, n)$ that has pair weights $(a_1, \ldots, a_{n - 1})$ and pair differences $(d_1, \ldots, d_{n - 1})$, and with diagonal factor of weight $d$.  Then $\gamma_{i}(f)$ is also a doubled path. Moreover, $\gamma_{i}(f)$ has the same pair weights and pair differences as $f$, except that the $i$-th pair weight of $\gamma_{i}(f)$ is $a_{i}+d$.
\end{prop}
\begin{proof}
The proof is very similar to that of Proposition~\ref{prop operation three}.  After applying the first $2(n - i - 1)$ Hurwitz moves, the factorization will have the following three factors in positions $2i - 1$, $2i$, and $2i + 1$:
\[
\Big(
 [(i\; i + 1); a_i], \;
 [(i\; i + 1); b_i], \; 
 [\id; (0, \ldots, 0, d, 0, \ldots, 0)]
\Big)
\]
where in the diagonal factor the nonzero weight is in position $i$.  Applying the middle four Hurwitz moves $\sigma_{2i}$, $\sigma_{2i - 1}$, $\sigma_{2i - 1}$, $\sigma_{2i}$ replaces these three factors with
\[
\Big(
 [(i\; i + 1); a_i + d], \;
 [(i\; i + 1); b_i + d], \; 
 [\id; (0, \ldots, 0, d, 0, \ldots, 0)]
\Big),
\]
and the final $2(n - i - 1)$ Hurwitz moves restore the suffix and place the factorization back in standard form.

For example, in the case $n= 3$, $i=1$, we have the following sequence of Hurwitz moves: 
\begin{align*}
f &= \Big(\br{[(1 2); a_1]}, \br{[(1 2); b_1]}, \bl{[(2 3);a_2]},\bl{[(2 3); b_2]}, \yg{[\id;(0,0,d)]}\Big)\\
&  \overset{\sigma_{4}^{-1}}{\longrightarrow}\Big( \br{[(1 2); a_1]}, \br{[(1 2); b_1]},\bl{[(2 3);a_2]},\yg{[\id;(0,d,0)]},\bl{[(2 3); b_2]} \Big)\\
& \overset{\sigma_{3}}{\longrightarrow} \Big(\br{[(1 2); a_1]}, \br{[(1 2); b_1]} ,\yg{[\id;(0,d,0)]},\bl{[(2 3);a_2+d]},\bl{[(2 3); b_2]} \Big)\\
& \overset{\sigma_{2}}{\longrightarrow} \Big( \br{[(1 2); a_1]},\yg{[\id;(0,d,0)]} ,\br{[(1 2); b_1+d]}, \bl{[(2 3);a_2+d]},\bl{[(2 3); b_2]} \Big)\\
& \overset{\sigma_{1}}{\longrightarrow}  \Big( \yg{[\id;(0,d,0)]},\br{[(1 2); a_1+d]} ,\br{[(1 2); b_1+d]},\bl{[(2 3);a_2+d]},\bl{[(2 3); b_2]} \Big)\\
& \overset{\sigma_{1}}{\longrightarrow}   \Big(\br{[(1 2); a_1+d]}, \yg{[\id;(d,0,0)]},\br{[(1 2); b_1+d]},\bl{[(2 3);a_2+d]},\bl{[(2 3); b_2]}\Big)\\
& \overset{\sigma_{2}}{\longrightarrow}  \Big( \br{[(1 2); a_1+d]} ,\br{[(1 2); b_1+d]},\yg{[\id;(0,d,0)]},\bl{[(2 3);a_2+d]},\bl{[(2 3); b_2]} \Big)\\
&\overset{\sigma_{3}^{-1}}{\longrightarrow}\Big(\br{[(1 2); a_1+d]}, \br{[(1 2); b_1+d]},\bl{[(2 3);a_2]},\yg{[\id;(0,d,0)]},\bl{[(2 3); b_2]}\Big)\\
&\overset{\sigma_{4}}{\longrightarrow}\Big(\br{[(1 2); a_1+d]} ,\br{[(1 2); b_1+d]},\bl{[(2 3);a_2]},\bl{[(2 3); b_2]}, \yg{[\id;(0,0,d)]}\Big)
= \gamma_{1}(f).
\end{align*}
In $\gamma_{1}(f)$, the pair weights are $(a_1+d, a_2, a_3)$ with the same pair differences as in $f$, as desired.
\end{proof}

By combining these operations, we give a sufficient condition for two doubled paths to belong to the same Hurwitz orbit.

\begin{prop}\label{prop h path in m}
Suppose $f$ and $f'$ are two doubled paths factoring the same element $g$ of weight $d$ in $G(m, p, n)$, with respective pair weights $(a_1, \ldots, a_{n - 1})$ and $(a'_1, \ldots, a'_{n - 1})$.
If there exists an $n \times (n - 1)$ $\ZZ$-matrix $M = (m_{ij})$ such that $m_{ij} = m_{ji}$ for $ i, j \in \{1, \ldots, n - 1\}$ and
\begin{equation}
\label{modular matrix relation}
\begin{pmatrix} a_{1} & \cdots & a_{n-1}\end{pmatrix}
+
\begin{pmatrix} d_1 &  \cdots & 
d_{n - 1} & d
\end{pmatrix} \cdot M \equiv
\begin{pmatrix}a'_{1} &  \cdots & a'_{n-1}\end{pmatrix} \pmod{m}
\end{equation}
(with equivalence taken coordinatewise), then $f$ and $f'$ belong to the same Hurwitz orbit.
\end{prop}

\begin{proof}
Suppose that \eqref{modular matrix relation} holds.  If $d = 0$ (i.e., the weight of $g$ is $0$, and $f$ and $f'$ do not contain loops), redefine $M$ by setting $m_{nj} = 0$ for $j = 1, \ldots, n - 1$ (and leaving all other $m_{ij}$ the same);\footnote{This step insures that the operation $\gamma_i$ is only applied if it is well-defined, i.e., if the factorizations have length $2n - 1$.} then \eqref{modular matrix relation} still holds for the redefined $M$.  Then it follows from Propositions~\ref{prop operation one}, \ref{prop operation three}, and~\ref{prop operation four} that applying the operations
\begin{itemize}
\item $\sigma_{2i - 1}^{m_{ii}}$ for $1 \leq i \leq n-1$, 
\item $\tau_{i, j}^{m_{ij}}$ for $1 \leq i < j \leq n - 1$, and 
\item $\gamma_i^{m_{nj}}$ for $1 \leq j \leq n-1$
\end{itemize}
to $f$ in any order produces a doubled path $\beta(f)$ that factors $g$ and has pair weights $(a'_1, \ldots, a'_{n - 1})$. As remarked just before Definition~\ref{def pair weight and constant}, a doubled path that factors a given diagonal element is determined by its pair weights; therefore, since $\beta(f)$ has the same pair weights as $f'$, it follows that $\beta(f) = f'$, and so $f$ and $f'$ belong to the same Hurwitz orbit.
\end{proof}

\begin{remark}
Although 
the actions of the various $\tau_{i,j}$ and $\gamma_k$ commute when restricted to act on doubled paths, they do not generally commute as elements of the braid group, and so their actions may not commute on other factorizations.
\end{remark}

Next, we rephrase Proposition~\ref{prop h path in m} in an easier-to-work-with form.

\begin{corollary}
\label{prop m}
Suppose 
$f$
and
$f'$
are two doubled paths factoring the same element $g = [\id; (k_1, \ldots, k_n)]$ in $G(m, p, n)$, with respective pair weights $(a_1, \ldots, a_{n - 1})$ and $(a'_1, \ldots, a'_{n - 1})$.  Let $r = \gcd(m, k_1, \ldots, k_n)$.  If $a_{j}\equiv a'_{j} \pmod{r}$ for $j = 1, \ldots, n - 1$, then $f$ and $f'$ belong to the same Hurwitz orbit.
\end{corollary}

\begin{proof}
Suppose that $f$ and $f'$ are two doubled paths in $G(m, p, n)$ that factor $[\id; (k_1, \ldots, k_n)]$, with respective pair weights $(a_1, \ldots, a_{n - 1})$ and $(a'_1, \ldots, a'_{n - 1})$, and that $a_{j}\equiv a'_{j} \pmod{r}$ for $j = 1, \ldots, n - 1$. By Proposition~\ref{prop h path in m}, it suffices to produce an $n \times (n - 1)$ $\ZZ$-matrix $M = (m_{ij})$ such that $m_{ij} = m_{ji}$ for $i,j\in \{1, \ldots, n-1\}$ when $i\neq j$, and
\[
\begin{pmatrix} a_{1} & \cdots & a_{n-1}\end{pmatrix}
+
\begin{pmatrix} d_1 &  \cdots & 
d_{n - 1} & d_n
\end{pmatrix} \cdot M \equiv
\begin{pmatrix}a'_{1} &  \cdots & a'_{n-1}\end{pmatrix} \pmod{m},
\]
where $d_i = k_1 + \ldots + k_i$ for $i = 1, \ldots, n$.
In order to construct such a matrix $M$, we introduce several auxillary sequences of integers.  We will repeatedly make use of the following fact (an extended form of the Euclidean algorithm): if $(q_1, \ldots, q_k)$ is any sequence of positive integers, then there exist integers $(p_1, \ldots, p_k)$ such that $p_1q_1 + \ldots + p_k q_k = \gcd(q_1, \ldots, q_k)$.  Now fix particular positive integer representatives $k_1, \ldots, k_n$ of their equivalence classes modulo $m$, so that $d_1, \ldots, d_n$ are also positive integers.

Since $d_i = k_1 + \ldots + k_i$ for $i = 1, \ldots, n$, we have $r = \gcd(m, d_1, \ldots, d_n)$.  Therefore, there exists a tuple $(x_0, x_1, \ldots, x_n)$ of integers such that 
\[
x_0m + x_1 d_1 + \ldots + x_n d_n = r,
\]
and consequently
\begin{equation}
\label{eq:x}
x_1 d_1 + \ldots + x_n d_n \equiv r \pmod{m}.
\end{equation}
Furthermore, since $a'_j - a_j \equiv 0 \pmod{r}$ for $j = 1, \ldots, n-1$, we can choose integers $(u_1, \ldots, u_{n-1})$ such that 
\begin{equation}
\label{eq:u}
a'_j - a_j \equiv u_j r \pmod{m}.
\end{equation}
For each $i, j$ such that $i,j\in\{1, \ldots, n-1\}$ and $i < j$, choose integers $y_{ij}$ and $y_{ji}$ such that 
\begin{equation}
\label{eq:y}
d_i y_{ji} - d_j y_{ij} = (u_j x_i - u_i x_j)\gcd(d_i, d_j).
\end{equation}
Negating both sides of \eqref{eq:y} and interchanging the letters $i$ and $j$, we have that for this choice of values $y_{ij}$, Equation~\eqref{eq:y} is valid for $i > j$, as well.

Now for $i = 1, \ldots, n$ and $j = 1, \ldots, n - 1$, define
\[
m_{ij} = \begin{cases}
u_j x_i + \dfrac{d_j}{\gcd(d_i, d_j)} y_{ij} & \text{if } i \neq j, \\[12pt]
\displaystyle
u_j x_j - \sum\limits_{\substack{k \colon k \neq j, \\ 1 \leq k \leq n}} \frac{d_k}{\gcd(d_k, d_j)} y_{kj} & \text{if } i = j.
\end{cases}
\]
We claim that the matrix $M = (m_{ij})$ satisfies the desired requirements.

First, if $i,j\in \{1, \ldots, n-1\}$ and $i\neq j$, we have
\[
m_{ij} = u_j x_i + \frac{d_j}{\gcd(d_i, d_j)} y_{ij} \qquad \text{ and } \qquad
m_{ji} = u_i x_j + \frac{d_i}{\gcd(d_i, d_j)} y_{ji}.
\]
By \eqref{eq:y}, these are equal to each other, fulfilling the first condition.

Second, by carrying out the matrix multiplication, we have that the $j$th coordinate of
$
\begin{pmatrix} d_1 &  \cdots & d_n \end{pmatrix} \cdot M
$
is
\begin{align*}
\sum_{i = 1}^n d_i m_{ij} 
& = \sum_{\substack{i=1, \ldots, n\\i\neq j}} \left(u_j x_i + \frac{d_j}{\gcd(d_i, d_j)}y_{ij}\right)d_i
+ \left(u_j x_j - \sum_{\substack{i=1, \ldots, n\\i\neq j}} \frac{d_i}{\gcd(d_i, d_j)}y_{ij}\right)d_j \\
& = u_j \left(\sum_{i=1}^{n} x_i d_i \right)\\
& \equiv u_{j} r \pmod{m},
\end{align*}
where in the last step we use \eqref{eq:x}.  Then by \eqref{eq:u}, it follows that $M$ satisfies \eqref{modular matrix relation}.

The result follows by Proposition~\ref{prop h path in m}.
\end{proof}

We end this section with a sufficient condition for two standard form factorizations of an arbitrary element in $G(m, p, n)$ to be Hurwitz equivalent.

\begin{lemma}\label{lem h path for all}
Suppose that $f$ and $f'$ are two standard form factorizations of the same element $g\in G(m, p, n)$.  Suppose further that $\Pi_{f}=\Pi_{f'}$ and that for every part $B$, the pair weights $(a_1, \ldots, a_{|B|-1})$ and $(a'_1, \ldots, a'_{|B|-1})$ of the doubled paths in $f|_{B}$ and $f'|_{B}$ satisfy the condition that $a_{i}\equiv a'_{i} \pmod{r(B)}$ for every $i\in \{1, \ldots, |B|-1\}$.  Then $f$ and $f'$ are Hurwitz-equivalent.
\end{lemma}

\begin{proof}
Let $f$ and $f'$ be as in the statement, with $\Pi \defeq \Pi_f = \Pi_{f'}$ the common cycle partition of $g$, and $B$ a part of $\Pi$.  Consider the restrictions $f|_B$ and $f'|_B$, two factorizations of the same element $g|_B$.  By Proposition~\ref{prop correct form}, $f|_{B}$ and $f'|_{B}$ both consist of a doubled path of length $k \defeq 2 |B| - 2$ or $2|B| - 1$, depending on the weight of $g|_B$, followed by a collection of factors whose corresponding edges in the factorization graph form a forest with $|B|$ components.

By Corollary~\ref{prop m}, there is a sequence $\beta_1$ of Hurwitz moves that act only on the first $k$ factors such that the first $k$ factors in $\beta_1(f'|_B)$ are equal to the corresponding factors in $f|_B$.  We now ``freeze'' these factors and work only with the forests.

By Proposition~\ref{prop:components commute}, if two factors correspond to edges in different components of the forest, they commute.  Consequently, we may choose an arbitrary ordering $C_1, \ldots, C_{|B|}$ on the cycles in $B$ and then choose braids $\beta_2$ and $\beta_3$ so that $\beta_2\beta_1(f'|_B)$ and $\beta_3(f|_B)$ consist first of $|C_1| - 1$ factors whose product is $C_1$, then $|C_2| - 1$ factors whose product is $C_2$, and so on.  Taking projections, the factors that factor a cycle generate the symmetric group on the support of the cycle; consequently, applying Theorem~\ref{Kluitmann}, there is another braid $\beta_4$ such that $\beta_4\beta_2\beta_1(f'|_B)$ and $\beta_3(f|_B)$ have the same factorization graph.  Moreover, it is easy to see that the weights of the edges in a factorization whose graph is a forest are determined by the product; since the two products are equal, it follows that $\beta_3^{-1}\beta_4\beta_2\beta_1(f'|_B) = f|_B$.

By Definition~\ref{def std form}(a), the factors in $f'|_B$ form a consecutive subsequence, so the Hurwitz moves in the braid $\beta_3^{-1}\beta_4\beta_2\beta_1$ may be lifted to Hurwitz moves acting on the full factorization $f'$.  Successively repeating this process for each part $B$ of $\Pi$ gives a Hurwitz path from $f'$ to $f$, as needed.
\end{proof}

\subsection{When are two standard form factorizations not Hurwitz-equivalent?}\label{subgroups}

The main result of this section is Lemma~\ref{lemma std in p}, which gives a necessary condition for two factorizations to lie in the same Hurwitz orbit (the converse of Lemma~\ref{lem h path for all}).  Our main tool is an invariant that distinguishes different Hurwitz orbits.

\begin{definition}
Given a factorization $f$ of an element in $G(m, n, p)$, denote by $G_f$ the subgroup of $G(m, n, p)$ generated by the factors in $f$.
\end{definition}

It is easy to see that $G_f$ is preserved by Hurwitz moves.  Our first result on the structure of $G_f$ is a completely straightforward consequence of Proposition~\ref{prop:components commute}.

\begin{prop}
\label{prop:components commute 2}
If $f$ is a reflection factorization that induces the cycle partition $\Pi_f$, we have that $G_f$ is the direct product of the restrictions $G_{f|_B}$ to individual parts $B$ of $\Pi_f$.
\end{prop}

Consequently, in what follows, we focus on factorizations $f$ for which the factorization graph $\Gamma_f$ is connected, so that the induced cycle partition $\Pi_f$ has only one part.  Also, by Lemma~\ref{lemma any in p}, it suffices to consider the case of factorizations in standard form.

\begin{prop}
\label{prop:isomorphic to a Gmpn}
Let $f$ be a connected standard form factorization of an element $g \in G(m, p, n)$, let $d$ be the weight of $g$, and let $r$ be as in Definition~\ref{def r star}.  Then $G_f \cong G\left(\frac{m}{r}, \frac{\gcd(m, d)}{r}, n\right)$.  More concretely, there is a diagonal element $\delta \in G(m, 1, n)$ such that the conjugation map $\phi_\delta$ defined by $\phi_\delta(g) = \delta g \delta^{-1}$ restricts to an isomorphism $\phi_\delta: G_f \overset{\sim}{\longrightarrow} G\left(\frac{m}{r}, \frac{\gcd(m, d)}{r}, n\right)$.
\end{prop}

\begin{proof}
First, consider the following sub-collection of edges of $\Gamma_f$: for each pair of factors in the doubled path with same underlying transposition, include the edge corresponding to the first, and also include all edges not in the doubled path.  Since $f$ is a connected standard form factorization, we have by Proposition~\ref{prop correct form} that these edges form a spanning tree of $\Gamma_f$.

Choose a linear order on the edges of the tree, as follows: the first edge is arbitrary, and each subsequently chosen edge should share an endpoint with a previously chosen edge.  Without loss of generality, let the edges in this order correspond to the reflections
\[
t_1 = [(i_1 \, j_1); a_1], \quad t_2 = [(i_2 \, j_2); a_2], \quad \ldots
\]
where for $k > 1$ we have that $i_k$ belongs to the set $\{i_1, j_1, j_2, \ldots, j_{k - 1}\}$ of previously chosen vertices, while $j_k$ does not.
Iteratively choose the weights of the diagonal entries of $\delta$, as follows: the $i_1$ entry has weight $0$, and for $k = 1, 2, \ldots$, the $j_k$ entry is chosen with the unique value so that $\delta t_k \delta^{-1}$ is a true transposition (i.e., $\delta t_k \delta^{-1} = [(i_k \, j_k); 0]$).  We claim this is the desired element.

We first show that $G_{\delta f \delta^{-1}} \subseteq G\left(\frac{m}{r}, \frac{\gcd(m, d)}{r}, n\right)$.  Let $c = \cyc(g)$.  By the choice of $\delta$, the first $2(c - 1)$ factors of $\delta f \delta^{-1}$ are
\begin{multline*}
\Big([(v_{1}\, v_{2}); 0], [(v_{1}\, v_{2});d_{1}], [(v_{2}\, v_{3}); 0], [(v_{2}\,v_{3}); d_{2}], \ldots,\\ [(v_{c-1}\, v_{c}); 0], [(v_{c-1} \, v_{c}); d_{c-1}]\Big),
\end{multline*}
and if the weight $d_c \defeq d$ of $g$ is nonzero then the $(2c - 1)$th factor is a diagonal reflection of weight $d$ in the $v_{c}$th position.  Multiplying out this prefix produces a diagonal element whose weight in position $v_i$ is $k_i \defeq d_1 + \ldots + d_i$ and whose other entries have weight $0$.  By Proposition~\ref{prop correct form} and the choice of $\delta$, the remaining factors of $\delta f \delta^{-1}$ form a forest of true transpositions, each component of which contains exactly one vertex $v_i$.  The product of such a set of factors is a permutation matrix such that each cycle contains exactly one of the $v_i$.  Consequently, in the product $\delta g \delta^{-1}$ of $\delta f \delta^{-1}$, the $k_i$ are precisely the cycle weights.  Since $d_1 = k_1$, $d_i = k_i - k_{i - 1}$ for $i > 1$, and $r = \gcd(m, k_1, \ldots, k_{c})$, it follows immediately that each $d_i$ is a multiple of $r$, and consequently that the weight of every entry appearing in all the factors in $\delta f \delta^{-1}$ belongs to $r\ZZ / m \ZZ \cong \ZZ/(m/r)\ZZ$.  Thus $G_{\delta f \delta^{-1}} \subset G(m/r, 1, n)$.  Moreover, the weights of all of the factors in $\delta f \delta^{-1}$ are not only multiples of $r$ but (stronger) multiples of $d$, i.e., they belong to $\gcd(m, d) \ZZ/ m\ZZ \cong \ZZ / (m/\gcd(m, d)) \ZZ$.  Thus in fact 
\[
G_{\delta f \delta^{-1}} 
\subset 
G\left(\frac{m}{r}, \frac{m/r}{m/\gcd(m, d)}, n\right) 
= 
G\left(\frac{m}{r}, \frac{\gcd(m, d)}{r}, n\right),
\]
as claimed.

For the reverse inclusion, we first note that $G\left(m/r, m/r, n\right) \subset G(m, 1, n)$ is generated over $\Symm_n$ by the reflection $[(12); r]$ \cite[\S2.7]{LehrerTaylor}.  The factors that make up the true transpositions in $\delta f \delta^{-1}$ generate the group $\Symm_n$ and so $\Symm_n \subseteq G_{\delta f \delta^{-1}}$.  By conjugating the other transposition-like factors in $\delta f \delta^{-1}$ by appropriate transpositions, we may produce the factors
\[
[(12); d_1], \quad [(12); d_2], \quad \ldots, \quad [(12); d_{c - 1}].
\]
If $d_c = d = 0$ then $[(12); d_c] \in \Symm_n$.  Otherwise, our factorization includes a diagonal reflection of weight $d_c$; conjugating it by an appropriate permutation produces $[\id; (d_c, 0, \ldots, 0)]$, and conjugating $[(12); 0]$ by this reflection produces $[(12); d_c]$.  Thus, in either case, we have in $G_{\delta f \delta^{-1}}$ the set $\{[(12); d_i] \colon i = 1, \ldots, c\}$.  Multiplying each of these reflections by $[(12); 0]$ produces the elements
\begin{equation}
\label{eq:penultimate factors}
[\id; (d_1, -d_1, 0, \ldots, 0)], \quad \ldots, \quad [\id; (d_n, -d_n, 0, \ldots, 0)].
\end{equation}
Since $r = \gcd(m, k_1, \ldots, k_c) = \gcd(m, d_1, \ldots, d_c)$, there exists a product of powers of the elements in \eqref{eq:penultimate factors} equal to $[\id; (r, -r, 0, \ldots, 0)]$, and multiplying by $[(12); 0]$ finally produces the desired reflection $[(12); r]$.  Consequently $G(m/r, m/r, n) \subseteq G_{\delta f \delta^{-1}}$.

If $d = 0$, the preceding arguments complete the proof.  If $d \neq 0$, then it suffices to observe that $G(m, p, n)$ is generated over $G(m, m, n)$ by any element of $G(m, 1, n)$ whose weight generates $p\ZZ/m\ZZ$.  Making the appropriate substitutions, since $G_{\delta f \delta^{-1}}$ contains $G(m/r, m/r, n)$ and contains an element of $G(m/r, 1, n)$ of weight $d$, it follows that $G(m/r, \gcd(m, d)/r, n) \subseteq G_{\delta f \delta^{-1}}$.  This completes the proof.
\end{proof}

Since diagonal matrices commute, the following consequence is immediate.

\begin{cor}
\label{cor:diagonal}
If $f$ is a connected standard form factorization of an element $g$ in $G(m, p, n)$, then the diagonal subgroup of $G_{f}$ is equal to the diagonal subgroup of $G\left(\frac{m}{r}, \frac{\gcd(m, d)}{r}, n\right)$.
\end{cor}

\begin{proposition}
\label{prop:same group implies modular condition}
Suppose that $f$ and $f'$ are two connected standard form factorizations of the same element $g \in G(m, p, n)$, having respective pair weights $(a_1, \ldots)$ and $(a'_1, \ldots)$.  If $G_f = G_{f'}$, then $a_i \equiv a'_i \pmod{r}$ for $i = 1, \ldots, \cyc(g) - 1$.
\end{proposition}
\begin{proof}
Let $c = \cyc(g)$.  By hypothesis, $f$ and $f'$ respectively contain factors
\[
[(v_1 \, v_2); a_1], \ldots, [(v_{c - 1} \, v_c); a_{c - 1}]
\qquad \text{ and } \qquad
[(v_1 \, v_2); a'_1], \ldots, [(v_{c - 1} \, v_c); a'_{c - 1}].
\]
For convenience, and without loss of generality, assume $v_i = i$ for $i = 1, \ldots, c$.  Then $G_f$ contains the element
\begin{multline*}
C \defeq [(1\cdots c); (a_1, \ldots, a_{c - 1}, -(a_1 + \ldots + a_{c - 1}), 0, \ldots, 0)] \\
 = [(1 2); a_1] \cdots [(c - 1 \, c); a_{c - 1}]
\end{multline*}
and similarly $G_{f'}$ contains the element
\[
C' \defeq [(1\cdots c); (a'_1, \ldots, a'_{c - 1}, -(a'_1 + \ldots + a'_{c - 1}), 0, \ldots, 0)]
.
\]
Since $G_f = G_{f'}$, we have $C, C' \in G_{f}$ and consequently
\[
C \cdot (C')^{-1} = [\id; (a_1 - a'_1, \ldots, a_{c - 1} - a'_{c - 1}, a'_1 + \ldots + a'_{c - 1} - a_1 - \ldots - a_{c - 1}, 0, \ldots, 0)]
\]
belongs to $G_f$ as well.  Since $C \cdot (C')^{-1}$ is diagonal, it follows from Corollary~\ref{cor:diagonal} that $C \cdot (C')^{-1} \in G(m/r, \gcd(m, d)/r, n)$.  Thus the weight of each entry in $C \cdot (C')^{-1}$ is a multiple of $r$.  In particular, $a_i - a'_i \equiv 0 \pmod{r}$, as claimed.
\end{proof}

We are now in position to give the necessary condition for two factorizations to belong to the same Hurwitz orbit.

\begin{lemma}\label{lemma std in p}
Suppose that $f$ and $f'$ are two standard form factorizations of the same element $g\in G(m, p, n)$.  Suppose further that $f$ and $f'$ are Hurwitz-equivalent.  Then $\Pi_{f}=\Pi_{f'}$.  Moreover, for every part $B$ in this partition, the pair weights $(a_1, \ldots, a_{|B|-1})$ and $(a'_1, \ldots, a'_{|B|-1})$ of the doubled paths in $f|_{B}$ and $f'|_{B}$ satisfy the condition that $a_{i}\equiv a'_{i} \pmod{r(B)}$ for every $i\in \{1, \ldots, |B|-1\}$.
\end{lemma}

\begin{proof}
It is easy to see that for any Hurwitz move $\sigma$, we have $G_f = G_{\sigma(f)}$.  Since $f$ and $f'$ are Hurwitz-equivalent, it follows that $G_f = G_{f'}$.  By Proposition~\ref{prop:components commute 2}, this group can be written as internal direct product over the components of the induced cycle partitions:
\begin{equation}
\label{eq:direct product}
G_f = \prod_{B \in \Pi_f} G_{f|_B} = \prod_{B' \in \Pi_{f'}} G_{f'|_{B'}}.
\end{equation}
It follows that the cycle partitions $\Pi_f$ and $\Pi_{f'}$ must be equal: indeed, if $e_i$ represents a standard basis vector for $\CC^n$, then by Proposition~\ref{prop:isomorphic to a Gmpn}, the orbit of the line $\CC e_i$ under $G_{f|_B}$ consists of all $\CC e_j$ such that $j$ is a vertex of $\Gamma_{f|_B}$, and so the components containing $i$ in $\Gamma_{f}$ and $\Gamma_{f'}$ have the same vertex set.  Thus the two partitions $\Pi_f$ and $\Pi_{f'}$ are equal.  Then it follows from \eqref{eq:direct product} that for each part $B$ of $\Pi_f = \Pi_{f'}$, we have $G_{f|_B} = G_{f'|_B}$.  Moreover, since the decomposition \eqref{eq:direct product} is a direct product, we can write $g = \prod_{B \in \Pi_f} g|_B$ uniquely for some elements $g|_B \in G_{f|_B} = G_{f'|B}$, and for each part $B$ we have that $f|_B$ and $f'|_B$ are reflection factorizations of $g|_B$.

Now fix a part $B$ of the common cycle partition $\Pi_f = \Pi_{f'}$, supported on a set of size $n'$.  Setting aside the (irrelevant) labeling of the vertices, $f|_B$ and $f'|_B$ are connected standard form factorizations of the same element $g|_B$ in $G_{f|_B} = G_{f'|_B} \subseteq G(m, 1, n')$.  Furthermore, any Hurwitz move respects the decomposition \eqref{eq:direct product}, so the hypothesis that $f$ and $f'$ are Hurwitz-equivalent implies that also the restrictions $f|_B$ and $f'|_B$ are Hurwitz-equivalent.  Then the result follows from Proposition~\ref{prop:same group implies modular condition}, applied separately to each part $B$.
\end{proof}

\subsection{Proofs of main theorem and corollaries}
\label{sec:proofs}

In this section, we assemble the work of the preceding sections to prove the first main result and its corollaries.  For convenience, we restate the results here.

\begin{thmm}
[\ref{main theorem}]
Given an element $g \in G(m, p, n)$, the number of Hurwitz orbits of its shortest factorizations is given by 
\[
\sum_{\Pi\in\Parmax(g)} \prod_{B\in \Pi} (r(B))^{|B|-1}.
\]
\end{thmm}

\begin{proof}
Fix an element $g$ in $G(m, p, n)$.
By Lemma~\ref{lemma any in p}, every Hurwitz orbit of minimum-length reflection factorizations of $g$ contains at least one factorization in standard form, so it suffices to determine when two standard form factorizations of $g$ belong to the same Hurwitz orbit.  By Proposition~\ref{prop correct number}, each standard form factorization $f$ of $g$ induces a maximum cycle partition $\Pi_f$ of $g$.  By Proposition~\ref{prop correct form}, if $f$ is in standard form and $B$ is a part of $\Pi_f$ then we may speak of \emph{the} doubled path in the restricted factorization $f|_B$.  By Lemmas~\ref{lem h path for all} and~\ref{lemma std in p}, two standard form factorizations $f$ and $f'$ belong to the same Hurwitz orbit if and only if (1) they induce the same cycle partition $\Pi \defeq \Pi_f = \Pi_{f'}$ and (2) for each part $B$ of $\Pi$, the sequence of $|B| - 1$ pair weights of the doubled paths of $f|_B$ and $f'|_B$ are congruent modulo $r(B)$.  The construction of Remark~\ref{rmk:construct a std factorization} guarantees that each of these congruence classes is realized by a standard form factorization.  Thus, there are $r(B)^{|B| - 1}$ inequivalent ways to choose the restricted factorization $f|_B$; these choices are independent for the different parts $B$ of $\Pi$, so there are $\prod_{B \in \Pi} r(B)^{|B| - 1}$ orbits that induce the cycle partition $\Pi$.  Summing over all maximum cycle partitions gives the result.
\end{proof}

\begin{corr}
[\ref{cor transitive}]
Let $g\in G(m,p,n)$. The shortest factorizations of $g$ form a single orbit under the Hurwitz action if and only if $\Parmax(g) = \{\Pi\}$ is a singleton set and either ${|B|} = 1$ or $r(B)=1$ for every $B\in\Pi$.
\end{corr}

\begin{proof}
Since $r(B)$ and $|B|$ are necessarily positive integers, the only way that $\sum_{\Pi\in\Parmax(g)} \prod_{B\in \Pi} (r(B))^{|B|-1}$ can be equal to $1$ is if there is only one summand, and each factor in the product is equal to $1$.  The result follows immediately.
\end{proof}

\begin{corr}
[\ref{cor single cycle}]
If $g\in G(m, p, n)$ has a single cycle, then the shortest factorizations of $g$ form a single orbit under the Hurwitz action.
\end{corr}

\begin{proof}
Let $g\in G(m, p, n)$ with a single cycle. Then there is only one cycle partition of $g$, whose unique part contains the unique cycle of $g$. The result follows from Corollary~\ref{cor transitive}.
\end{proof}

\begin{corr}
[\ref{cor special}]
Let $g\in G(m,1,n)$.  Then the shortest factorizations of $g$ form a single orbit under the Hurwitz action if and only if $g$ does not have two cycles of nonzero weight whose weights sum to $0$.
\end{corr}
\begin{proof}
Fix an element $g\in G(m,1,n)$ with cycle weights $(k_1, \ldots, k_{\cyc(g)})$ and a cycle partition $\Pi$ of $g$.  If any part $B$ of $\Pi$ contains three or more cycles, we may form a new cycle partition $\Pi'$ by splitting $B$ into $|B|$ singleton parts.  Since the part $B$ contributes at most $2$ to $v(\Pi)$, it follows that $v(\Pi') \geq 3 + (v(\Pi) - 2) > v(\Pi)$, so $\Pi$ is not maximum.  Similarly, if $\Pi$ contains a part with exactly two cycles whose weights do not sum to zero, or a part with two cycles of weight zero, splitting the part into singletons increases the value.  On the other hand, any partition $\Pi$ in which all parts consist either of a single cycle or of two cycles whose weights are nonzero but add to zero has $v(\Pi) = \cyc(g) + \#\{i \colon k_i = 0\}$ and so is a maximum cycle partition of $g$.

It follows from the preceding paragraph that if $g$ has two cycles of nonzero weight whose weights sum to zero, then $g$ has at least two maximum cycle partitions (the partition into all singletons; or, group those two cycles together) and so more than one Hurwitz orbit of shortest reflection factorizations (by Corollary~\ref{cor transitive}).  Conversely, if $g$ does not have two such cycles, then it follows from the preceding paragraph that the unique maximum cycle partition of $g$ is the partition into singleton parts, and so $g$ has a unique Hurwitz orbit of shortest reflection factorizations (again by Corollary~\ref{cor transitive}).  This completes the proof.
\end{proof}

\section{A reformulation in terms of invariants}
\label{sec:invariants}

In this section, we reformulate the work in Section~\ref{sec:main} to give a criterion to tell when two minimum-length reflection factorizations of an arbitrary element in $G(m, p, n)$ belong to the same Hurwitz orbit.  We then discuss the extent to which this can be extended to the other complex reflection groups. Recall that for a factorization $f$, we denote by $G_f$ the group generated by the factors in $f$.

\begin{theorem}\label{thm:invariants}
Let $G = G(m, p, n)$ and let $g$ be any element of $G$.  Two minimum-length reflection factorizations $f$ and $f'$ of $g$ lie in the same Hurwitz orbit if and only if $G_f = G_{f'}$.
\end{theorem}

\begin{proof}
It has already been observed that the group $G_f$ is preserved by the Hurwitz action. 
Conversely, suppose $G = G(m, p, n)$ and $g \in G$.  Consider two minimum-length factorizations $f$ and $f'$ of $g$ such that $G_f = G_{f'}$.  By Proposition~\ref{lemma any in p}, we may as well assume that $f$ and $f'$ are in standard form.  By Proposition~\ref{prop:components commute 2}, it follows that $f$ and $f'$ induce the same cycle partition $\Pi$ of $g$.  By Proposition~\ref{prop:same group implies modular condition}, it follows that for each part $B$ of $\Pi$, the pair weights of the doubled paths in the restricted factorizations $f|_B$ and $f'|_B$ are congruent modulo $r(B)$.  Finally, by Lemma~\ref{lem h path for all}, it follows that $f$ and $f'$ are Hurwitz-equivalent.
\end{proof}

\begin{remark}
\label{rem:BGRW invariants}
In \cite[Cor.~1.3]{BGRW}, Baumeister--Gobet--Roberts--Wegener prove that if $G$ is any Coxeter group and $f$ is a minimum-length reflection factorization of an element $g \in G$ such that $G_f$ is finite, then all minimum-length $G_f$-reflection factorizations of $g$ belong to a single Hurwitz orbit.  This immediately implies the hard direction of Theorem~\ref{thm:invariants} when $G$ is a finite \emph{real} reflection group.
\end{remark}

In what follows, we consider the exceptional complex reflection groups.  In the Shephard--Todd indexing scheme, these are named $G_4, G_5, \ldots, G_{37}$; the groups $G_{23}$, $G_{28}$, $G_{30}$, $G_{35}$, $G_{36}$, and $G_{37}$ are the exceptional \emph{real} reflection groups of types $\type{H}_3$, $\type{F}_4$, $\type{H}_4$, $\type{E}_6$, $\type{E}_7$, and $\type{E}_8$, respectively.  Most of the exceptional groups ($G_4$ through $G_{22}$) are rank $2$, meaning they act irreducibly on a two-dimensional space; the largest non-real exceptional group $G_{34}$ has rank $6$ and cardinality $39191040$.

Suppose that $G$ is an arbitrary complex reflection group, $g \in G$, and $f$ and $f'$ are two minimum-length reflection factorizations of $g$ that generate the same subgroup $G_f = G_{f'}$.  By exhaustive computation in the non-real examples, we have verified that, if $G_f$ is isomorphic to any of the exceptional groups $G_6$, $G_7$, $G_9$, $G_{11}$, $G_{12}$, $G_{13}$, $G_{14}$, $G_{15}$, $G_{19}$, $G_{21}$, $G_{22}$, $G_{23}$, $G_{24}$, $G_{27}$, $G_{28}$, $G_{29}$, $G_{30}$, $G_{31}$, $G_{33}$, $G_{35}$, $G_{36}$, or $G_{37}$,
then it follows that $f$ and $f'$ are Hurwitz-equivalent.  The next remarks discuss the situation in the remaining exceptional complex reflection groups, all of which contain some element $g$ with multiple Hurwitz orbits of shortest factorizations that generate the whole group.

\begin{example}
\label{rem:G16}
Consider\footnote{We are grateful to Theodosius Douvropoulos for this example.} the exceptional group $G = G_{16}$, which is generated by two reflections $a, b$ subject to the relations $a^5 = b^5 = 1$, $aba = bab$.  The group may be represented by matrices so that both generators have non-unit eigenvalue $\omega = \exp(2\pi i/5)$.  The element $g \defeq a^2 b^3$ is not a reflection (e.g., because it has determinant $1$) and so $f_1 = (a^2, b^3)$ is a shortest reflection factorization of $g$.  Clearly $f_1$ generates $G$.  It is slightly more work to check that $f_2 = (a^{-1}ba, b^{-2} a^{-1} b^2)$ is another shortest reflection factorization of $g$
that also generates $G$.
The factors in $f_1$ have determinants $\omega^2$ and $\omega^3$, while those in $f_2$ have determinants $\omega$ and $\omega^4$, so they lie in different conjugacy classes.  However, applying a Hurwitz move to a factorization $f$ does not change its multiset of $G_f$-conjugacy classes: the multiset of factors changes by one element, and the element that changes is conjugated by a reflection in $G_f$.  Thus $f_1$ and $f_2$ lie in different Hurwitz orbits.
\end{example}

\begin{remark}
\label{rem:counter-examples}
In \cite{JBL}, it was conjectured that the two invariants $G_f$ and the multiset of $G_f$-conjugacy classes of factors were sufficient to distinguish factorizations of every element in \emph{every} complex reflection group.\footnote{
We mention the earlier work that informed this conjecture: for symmetric groups, it can be proved by extending the work of Kluitmann~\cite{Kluitmann} to cover factorizations with disconnected graphs, as in \cite{BIT}.  It was proved for dihedral groups by Berger \cite{Berger11}, and for the ``tetrahedral family'' of rank-$2$ complex reflection groups ($G_4$, $G_5$, $G_6$, $G_7$) by T.\ Minnick, C.\ Pirillo, S.\ Racile and E.\ Wang (unpublished; personal communication).  Theorem~\ref{thm:invariants} and Remark~\ref{rem:BGRW invariants} establish the case of minimum-length factorizations in $G(m, p, n)$ and in real reflection groups, respectively.}
By brute-force computer calculations (using SageMath \cite{Sage}, including its interface with GAP \cite{GAP} and Chevie \cite{Chevie}), we discovered that the conjecture \emph{fails} in $G_{32}$ and $G_{34}$. 
We briefly describe these counter-examples now.

The group $G_{32}$ is generated by four reflections $a, b, c, d$ subject to the relations $a^3 = b^3 = c^3 = d^3 = 1$, $aba = bab$, $bcb = cbc$, $cdc = dcd$, $ac = ca$, $ad = ca$, and $bd = db$.  The element $ab^{-1}cdabcab^{-1}a^{-1}cdbc$ has reflection length $5$.  Its 
minimum-length reflection factorizations fall into five Hurwitz orbits.
The factorizations
\begin{align*}
& ( a b c^{-1} d c b^{-1} a^{-1}, \;
a^{-1} b c d c^{-1} b^{-1} a, \;
b c d c^{-1} b^{-1}, \;
b a^{-1} b c b^{-1} a b^{-1}, \;
b^{-1} c d c^{-1} b
), \\
&(
b a^{-1} b c d c^{-1} b^{-1} a b^{-1}, \;
b c b^{-1}, \;
a b a^{-1}, \;
b^{-1} c^{-1} d c b, \;
a b c b^{-1} a^{-1}
), \\
\intertext{and}
&(b^{-1} c b, \;
a^{-1} b c b^{-1} a, \;
b c^{-1} d c b^{-1}, \;
a b^{-1} c^{-1} d c b a^{-1}, \;
b),
\end{align*}
lie in three different orbits, 
but all generate the full group $G_{32}$ and all have the same multiset of conjugacy classes.  
The only other minimum-length counter-examples to the conjecture in $G_{32}$ are the conjugates of this element and of its inverse.

There are also counter-examples in $G_{34}$: among the elements of reflection length at most $7$, there are two (mutually inverse) conjugacy classes of examples, of index $128$ and $145$ in the list of conjugacy classes generated by the command \verb|(ReflectionGroup(34)).conjugacy_classes_representatives()| in SageMath: these elements have three Hurwitz orbits of factorizations, two of which generate the whole group $G_{34}$. (All reflections in $G_{34}$ are conjugate to each other.)  It is also possible that there are other examples in $G_{34}$: the group contains fourteen conjugacy classes of elements of reflection length $8$, two of reflection length $9$, and two of reflection length $10$ (a total of nine mutually inverse pairs), but because the numbers of factorizations for these elements are enormous (they range from $705438720$ for the class with index 98 in the SageMath indexing 
to $42664933785600$ for the class with index 87
), it seems infeasible to test computationally whether they satisfy the conjecture.
\end{remark}

For the question of whether there is a uniform version of Theorem~\ref{thm:invariants}, see Section~\ref{sec:open}.


\section{The quasi-Coxeter property}
\label{sec:qc}

In the case of \emph{real} reflection groups (that is, finite Coxeter groups), the transitivity of the Hurwitz action on factorizations of a given element is closely tied to the \emph{quasi-Coxeter property}, which we define now.\footnote{The term \emph{quasi-Coxeter element} was originally defined by Voigt \cite{Voigt}, with a slightly broader meaning than the one given here -- see \cite[Rmk.~1.8]{BGRW}.  Our definition of a ``weak quasi-Coxeter element'' coincides with the definition of a ``quasi-Coxeter element'' from \cite{BGRW}.}

\begin{definition}
Let $G$ be a finite complex reflection group, and $g$ an element of $G$.  We say that $g$ is a \emph{weak quasi-Coxeter element} for $G$ if there is a shortest reflection factorization of $g$ that generates $G$.  If in fact every shortest reflection factorization of $g$ generates $G$, we say that $g$ is a \emph{strong quasi-Coxeter element} for $G$.
\end{definition}

The following theorem of Baumeister--Gobet--Roberts--Wegener, promised in the introduction, explains the connection.

\begin{theorem}[{part of \cite[Thms.~1.1 and 1.2]{BGRW}}]
\label{thm:BGRW qc}
If $w$ is a weak quasi-Coxeter element for a finite real reflection group $W$, then the Hurwitz action is transitive on the shortest reflection factorizations of $w$, and consequently $w$ is a strong quasi-Coxeter element for $W$.

Conversely, if the Hurwitz action is transitive on shortest reflection factorizations of $w$, then there is a parabolic subgroup (defined below) $W'$ of $W$ such that $w$ is a strong quasi-Coxeter element for $W'$.
\end{theorem}

Here a \emph{parabolic subgroup} is a subgroup that pointwise fixes a subspace of the space $V$ on which $W$ acts; in the case of a finite real reflection group, one may equivalently \cite[\S5-2]{Kane} say that it is a conjugate of a \emph{standard parabolic subgroup} generated by a subset of the standard Coxeter generating set.

In this section, we explore the extent to which Theorem~\ref{thm:BGRW qc} is valid in the complex setting.  We begin by characterizing the weak quasi-Coxeter elements in the group $G(m, p, n)$.

\begin{lemma}
\label{lem:qc}
Let $G = G(m, p, n)$.  Then $g \in G$ has a shortest factorization $f$ such that $G_f = G$ if and only if the following conditions hold:
  \begin{enumerate}[(i)]
  \item the cycle weights of $g$ generate $\ZZ / m\ZZ$,
  \item the weight of $g$ generates $p\ZZ / m\ZZ$, and
  \item no nontrivial subset of the cycles of $g$ has weight $0$ modulo $p$.
  \end{enumerate}
\end{lemma}
(If $p = m$, then the second condition is vacuous.)
\begin{proof}
First, suppose that $g \in G = G(m, p, n)$ has a shortest factorization $f$ that generates $G$.

By Lemma~\ref{lemma any in p}, we may assume without loss of generality that $f$ is in standard form.
Since $G_f = G$, we have by Proposition~\ref{prop:components commute 2} that $\Pi_f$ must be the one-part partition of the cycles of $g$.  
Since $\Pi_f$ is the one-part partition, $f$ is connected.
Let $B$ be the unique part of $\Pi_f$ and let $r = r(B)$.  By Proposition~\ref{prop:isomorphic to a Gmpn}, $G_f \cong G\left(\frac{m}{r}, \frac{\gcd(m, d)}{r}, n\right)$, where $d$ is the weight of $g$.  Since $G_f = G(m, p, n)$, it follows that $r = 1$, and so the cycle weights of $g$ generate $\ZZ/m\ZZ$ (condition (i)), and that $\gcd(m, d) = p$, so that the weight of $g$ generates $p\ZZ / m\ZZ$ (condition (ii)).  To complete this direction, we must show that condition (iii) holds, i.e., that no nontrivial subset of the cycles of $g$ has weight $0$ modulo $p$.  We consider separately the case $p = m$ and $p < m$.

Assume $p = m$.  As observed in Remark~\ref{rmk:shi special cases}, in this case $v(\Pi) = 2 |\Pi|$ for any cycle partition $\Pi$; in particular, $v(\Pi_f) = 2$.  Since $f$ is a shortest factorization of $g$, we have by Proposition~\ref{prop correct number} that the partition $\Pi_f$ is maximum.  Assume for contradiction that there is a nontrivial subset $S$ of the cycles of $g$ such that $\wt(S) = 0$.  Since $\wt(g) = 0$, it follows that $\Pi \defeq \{S, \overline{S}\}$ is a cycle partition of $g$ (where $\overline{S}$ denotes the complement of $S$), having value $v(\Pi) = 4 > 2 = v(\Pi_f)$.  This contradicts the fact that $\Pi_f$ is maximum, so in fact no nontrivial subset of the cycles may have weight $0$, and condition (iii) holds in this case.

Now assume $p < m$.  Since the weight of $g$ generates $p\ZZ/m\ZZ$, it must be nonzero.  Thus $v(\Pi_f) = 1$.  Since $f$ is a shortest factorization of $g$, we have by Proposition~\ref{prop correct number} that the partition $\Pi_f$ is maximum.  Assume for contradiction that there is a nontrivial subset $S$ of the cycles of $g$ such that $\wt(S) = 0 \pmod{p}$.  Since $\wt(g) = 0 \pmod{p}$, it follows that $\Pi \defeq \{S, \overline{S}\}$ is a cycle partition of $g$ (where $\overline{S}$ denotes the complement of $S$), having value $v(\Pi) \geq 2 > 1 = v(\Pi_f)$.  This contradicts the fact that $\Pi_f$ is maximum, so in fact no nontrivial subset of the cycles may have weight $0$ modulo $p$, and condition (iii) holds in this case.

Conversely, suppose that $g \in G(m, p, n)$ satisfies the three given conditions.  Let $f$ be a shortest reflection factorization of $g$ and let $d = \wt(g)$.  By Lemma~\ref{lemma any in p}, we may assume without loss of generality that $f$ is in standard form.  By condition (iii), $g$ has a unique cycle partition $\Pi$, with all cycles in the same part.  Since $\Pi$ is unique, it must be the case that $\Pi_f = \Pi$, and consequently $f$ is connected.  Let $B$ be the unique part of $\Pi_f$ (containing all cycles of $g$), and let $r = r(B)$.  By condition (i), we have $r(B) = 1$.  By condition (ii), we have $\gcd(m, d) = p$.  Therefore, by Proposition~\ref{prop:isomorphic to a Gmpn}, we have 
\[
G_f \cong G\left(\frac{m}{r}, \frac{\gcd(m, d)}{r}, n\right) = G(m, p, n).
\]
Since $G_f \subseteq G(m, p, n)$, it follows that actually $G_f = G(m, p, n)$, as desired.
\end{proof}

As a consequence of Lemma~\ref{lem:qc}, we can directly extend the first half of Theorem~\ref{thm:BGRW qc} to the group $G(m, p, n)$.

\begin{corollary}
\label{cor:qc implies transitive}
Suppose that $g$ is a weak quasi-Coxeter element for $G(m, p, n)$. Then (a) the Hurwitz action is transitive on shortest reflection factorizations of $g$, and (b) $g$ is a strong quasi-Coxeter element for $G(m, p, n)$.
\end{corollary}
\begin{proof}
Suppose $g$ is a weak quasi-Coxeter element for $G(m, p, n)$.  By Lemma~\ref{lem:qc}(iii), the unique cycle partition $\Pi$ of $g$ is the one-part partition.  For this partition, with unique part $B$ containing all cycles of $g$, we have by Lemma~\ref{lem:qc}(i) that $r(B) = 1$.  Therefore, by Corollary~\ref{cor transitive}, all shortest reflection factorizations of $g$ belong to a single Hurwitz orbit.  This establishes (a).  Since the group generated by a factorization is preserved under Hurwitz moves, it follows from (a) that all shortest factorizations of $g$ generate the same subgroup.  Since $g$ is weak quasi-Coxeter, this subgroup is the whole group $G(m, p, n)$.  This establishes (b).
\end{proof}

As in the case of Theorem~\ref{thm:invariants}, the same result does not hold if we replace $G(m, p, n)$ with an arbitrary complex reflection group: Example~\ref{rem:G16} shows that the statement of Corollary~\ref{cor:qc implies transitive}(a) fails for $G_{16}$, while the next example shows that even the weaker (b) fails in general.

\begin{example}
Consider the exceptional group $G = G_{10}$, which is generated by two reflections $a, b$ subject to the relations $a^3 = b^4 = 1$, $abab=baba$.  
Let $g = ba^{-1}ba^{-1}$.  One can show that $g$ has reflection length $3$.  One shortest reflection factorization of $g$ is
$
(b, a, aba^{-1})
$,
and this triple obviously generates the entire group $G$.  Thus, $g$ is weak quasi-Coxeter for $G$.  However, another shortest factorization for $g$ is
$(ba^{-1}b^{-1}, b^{-1}, a^{-1})$.
This triple generates an index-$2$ subgroup of $G$ isomorphic to $G_7$.  Thus $g$ is not strong quasi-Coxeter for $G$.
\end{example}

For an arbitrary complex reflection group $G$, define the \emph{rank} $\rank(G)$ to be the dimension of the subspace of $V$ on which $G$ acts nontrivially (i.e., the vectors that are not fixed by every $g$ in $G$).  Since each reflection fixes a hyperplane, it is clear that every generating set of reflections must have size at least $\rank(G)$.
In the preceding examples, the elements under consideration have reflection length strictly larger than the rank of the group they belong to.  The next result, a uniform version of Corollary~\ref{cor:qc implies transitive}(b), shows that this is not a coincidence.

\begin{theorem}
\label{cor:qc length n}
If $G$ is any finite complex reflection group and $g$ is a weak quasi-Coxeter element for $G$ whose reflection length is equal to the rank of $G$, then $g$ is a strong quasi-Coxeter element for $G$.
\end{theorem}
\begin{proof}
First, we claim that the truth of the statement for all complex reflection groups follows from its truth for irreducible groups.  To this end, suppose $G = X \times Y$ is reducible.  Then every reflection factorization $f$ of an element $g = (x, y)$ of $G$ is the result of shuffling together a ($X \times \{\id_Y\}$)-reflection factorization $f_X$ of $(x, 1)$ and a ($\{\id_X\} \times Y$)-reflection factorization $f_Y$ of $(1, y)$.  
Now suppose $f$ is a minimum-length reflection factorization of $g$ such that $G_f = G$ and $\#f = \rank(G)$.  Since the factors in $f_X$ commute with those in $f_Y$, we have $G = G_f = G_{f_X} \times G_{f_Y}$.  It follows that $x$ is a weak quasi-Coxeter element for $X$ and $y$ is a weak quasi-Coxeter element for $Y$, and so also that $\lR(x) \geq \rank(X)$ and $\lR(y) \geq \rank(Y)$.  On the other hand, we have by hypothesis that
\[
\rank(X) + \rank(Y) = \rank(G) = \lR(g) = \lR(x) + \lR(y),
\]
so in fact $\lR(x) = \rank(X)$ and $\lR(y) = \rank(Y)$.  By induction on the number of irreducible components of $G$, we may assume that $x$ is strong quasi-Coxeter for $X$ and $y$ is strong quasi-Coxeter for $Y$, and consequently $G$ is strong quasi-Coxeter for $G$.  Therefore, it suffices to check the result for irreducible groups.

We now proceed case-by-case.  The infinite family is covered by Corollary~\ref{cor:qc implies transitive}, and the real exceptional groups are covered by Theorem~\ref{thm:BGRW qc}.  For each non-real exceptional group, the check is a finite computation; however, because of the size of some of the groups involved, it is not a \emph{trivial} check.  We describe our computational approach, which we implemented on SageMath \cite{Sage} using its interface with GAP \cite{GAP} and Chevie \cite{Chevie}; the code to carry out the check is attached to the arXiv version of this paper as an ancilliary file.

All the properties in the theorem statement are invariant under conjugation, so it suffices to check a set of conjugacy class representatives.

Given a conjugacy class representative $g$, we first confirm that its reflection length is equal to the rank of the group.
We use a standard technique based on the character theory of the group $G$ (as in, e.g., \cite{ChapuyStump}) to compute the total number of minimum-length reflection factorizations of $g$.

Next, we produce a reflection factorization of $g$ by successively testing reflections $r$ to see if $\lR(r^{-1}g) < \lR(g)$.  If so, $g$ has a shortest factorization that begins $(r, \ldots)$, and we proceed recursively.  Having produced a factorization, we construct its Hurwitz orbit by applying Hurwitz moves one by one, discarding already-discovered factorizations.  If this orbit does not exhaust the minimum-length reflection factorizations, we find a new factorization (not in any previously-produced Hurwitz orbit) and continue until the sum of the sizes of the Hurwitz orbits produced is equal to the number of minimum-length reflection factorizations.

If $g$ has only one Hurwitz orbit of minimum-length factorizations then all its factorizations generate the same subgroup and the result is immediate.  Otherwise, we check for each orbit whether or not its factorizations generate the whole group $G$.  

The result of this calculation was to verify the claim in all the non-real irreducible complex reflection groups.
\end{proof}

The condition that the reflection length of $g$ be equal the rank of $G$ is quite natural: in a real reflection group $W$ of rank $n$, it follows from Carter's theorem \cite[Lem.~2]{Carter} that $\lR(g) \leq n$ for all $g \in W$, and so all quasi-Coxeter elements have reflection length $n$.  Moreover, in a complex reflection group of rank $n$, Coxeter elements (if any exist) all have reflection length $n$ \cite[\S7.1]{Bessis}.  Thus, it is natural to ask what Lemma~\ref{lem:qc}
says in the case of elements of reflection length $n$.

For comparison, we mention the known descriptions of quasi-Coxeter and Coxeter elements in combinatorial groups.  In type $\type{B}_n$ (the group $G(2, 1, n)$), the quasi-Coxeter elements are precisely the Coxeter elements; combinatorially, these are the $n$-cycles of weight $1$ \cite[Lem.\ 6.4 and its proof]{BGRW}.  In type $\type{D}_n$ (the group $G(2, 2, n)$), the Coxeter elements have cycle type $(n - 1, 1)$, with both cycles of weight $1$, while the quasi-Coxeter elements include all elements with two cycles, both of weight $1$ \cite[Rem.\ 8.3]{BGRW}.  Coxeter elements are defined for the \emph{well-generated} complex reflection groups (those with a generating set of reflections whose cardinality is equal to the rank of the group) \cite{RRS, Bessis}.  Among the groups $G(m, p, n)$, the well-generated ones are $G(m, 1, n)$, whose Coxeter elements are the $n$-cycles of primitive weight modulo $m$, and $G(m, m, n)$, whose Coxeter elements are the elements of cycle type $(n - 1, 1)$ in which both cycles have primitive weight modulo $m$.

\begin{corollary}
In $G(m, 1, n)$, an element is quasi-Coxeter if and only if it has a single cycle and its weight is primitive modulo $m$ (i.e., if and only if it is a Coxeter element).  All such elements have reflection length $n$.

In $G(m, m, n)$, an element is quasi-Coxeter of reflection length $n$ if and only if it has exactly two cycles and their weights are primitive modulo $m$.

If $1 < p < m$, the group $G(m, p, n)$ does not contain any quasi-Coxeter elements of reflection length $n$.
\end{corollary}
\begin{proof}
Suppose $g$ is a quasi-Coxeter element for $G(m, 1, n)$.  Since $p = 1$, every subset of cycles of $g$ has weight $0$ modulo $p$.  Then by Lemma~\ref{lem:qc}(iii), $g$ must have only one cycle.  By Lemma~\ref{lem:qc}(i), the weight of that cycle must be primitive modulo $m$.  This completes one direction; the converse is completely straightforward by Lemma~\ref{lem:qc}.  Furthermore, by Theorem~\ref{shi}, every such element has reflection length $n$.

Suppose instead that $g$ is a quasi-Coxeter element for $G(m, m, n)$ of reflection length $n$.  By Lemma~\ref{lem:qc}(iii), the only cycle partition of $g$ is the one-part partition, having value $2$.  By Theorem~\ref{shi}, we have $n = \lR(g) = n - \cyc(g) + 2$, and consequently $g$ has exactly two cycles.  Since $g \in G(m, m, n)$, these cycle weights sum to $0$, so each generates the same subgroup of $\ZZ/m\ZZ$.  By Lemma~\ref{lem:qc}(i), the cycle weights are primitive modulo $m$.  This completes one direction; the converse is completely straightforward.

Finally, if $1 < p < m$ then $G(m, p, n)$ is not generated by any set of $n$ reflections.  Therefore, if $g$ has reflection length $n$, none of its shortest factorizations generate $G(m, p, n)$.  Consequently $G(m, p, n)$ has no quasi-Coxeter elements of reflection length $n$.
\end{proof}

\section{Open problems}
\label{sec:open}

We end with some natural problems left open by the preceding work.

\subsection{Uniform counting?}

Is a uniform version of Theorem~\ref{main theorem}?  That is, can the quantities that appear in the statement be given an interpretation that generalizes to an arbitrary complex reflection group?  One would hope for some sort of geometric interpretation, especially in the real case.

\subsection{Missing invariants?}

Remark~\ref{rem:counter-examples} shows that the invariants \[
\left(\text{product } g, \text{generated subgroup } H, \text{multiset of } H\text{-conjugacy classes}\right)
\]
fail to distinguish Hurwitz orbits of minimum-length factorizations in the groups $G_{32}$ and $G_{34}$ (but not in any other irreducible complex reflection group).  Is there some other natural invariant that can distinguish the Hurwitz orbits in these cases, as well?

In the real case, it follows from \cite[Cor.~1.3]{BGRW} that just the invariants $g$ and $H$ suffice for shortest factorizations.  All known proofs of this fact are case-by-case, and it would be highly desirable to give a uniform proof.

\subsection{Special subgroups for transitive elements?}

If the Hurwitz action is transitive on the shortest reflection factorizations of an element $g$, then $g$ is a strong quasi-Coxeter element for the subgroup generated by these factorizations.  In contrast with the real case (Theorem~\ref{thm:BGRW qc}), the subgroups that arise this way are not limited to the parabolic subgroups, even when restricting attention to $G(m, p, n)$.
\begin{example}
For any $n > 1$, $m > 2$, consider $g = [\id; (1, \ldots, 1)] \in G = G(m, 1, n)$ (an element of the center of $G$).  By Remark~\ref{rmk:shi special cases} we have $\lR(g) = n$,
and so one shortest factorization is as a product of the $n$ diagonal reflections of weight $1$.
By Corollary~\ref{cor special}, the Hurwitz action is transitive on shortest reflection factorizations of $g$.  Thus $g$ is a strong quasi-Coxeter element for the diagonal subgroup $(\ZZ/m\ZZ)^n$ in $G$.  But $G$ is not a parabolic subgroup because no nonzero vector is fixed by $g$.
\end{example}

Is there a good description of the subgroups of a complex reflection group that arise in this way?

\subsection{Uniform proof?}

Theorem~\ref{cor:qc length n} is valid for any complex reflection group, but its proof relies heavily on the classification and brute-force checks in the exceptional cases.  Is it possible to give a uniform proof that ``explains'' why the result should be true?

\end{document}